\documentclass{article}
\usepackage[text={160mm,232mm},centering]{geometry}
\usepackage{amsmath,amssymb,amsthm,mathabx,verbatim}
\usepackage[
backend=biber,
style=alphabetic,
maxalphanames=4,
maxbibnames=4,
isbn=false,
url=false,
giveninits=true,
doi=true,
]{biblatex}
\usepackage[bookmarks=true]{hyperref}
\usepackage{enumitem}
\usepackage{cleveref}
\usepackage{graphicx}
\usepackage{xcolor}
\usepackage{soul}
\usepackage{float}
\usepackage{longtable}

\newcommand{\tx}{\textnormal}
\newcommand{\ol}{\overline}

\newcommand{\QQ}{\mathbb Q}
\newcommand{\ZZ}{\mathbb{Z}}
\newcommand{\CC}{\mathbb{C}}
\newcommand{\RR}{\mathbb R}
\newcommand{\FF}{\mathbb F}
\newcommand{\HH}{\mathbb{H}}
\newcommand{\PP}{\mathbb P}
\newcommand{\bT}{\mathbb T}
\newcommand{\calO}{\mathcal O}
\newcommand{\calN}{\mathcal N}

\newcommand{\smt}[4]{\left( \begin{smallmatrix} #1 &#2\\ #3 &#4\\ \end{smallmatrix} \right) } 
\newcommand{\End}{\mathrm{End}}
\newcommand{\lcm}{\mathrm{lcm}}
\newcommand{\Gal}{\mathrm{Gal}}
\newcommand{\Norm}{\mathrm{Norm}}
\newcommand{\Aut}{\mathrm{Aut}}
\newcommand{\GL}{\mathrm{GL}}
\newcommand{\PGL}{\mathrm{PGL}}

\newcommand{\GLn}{\mathrm{GL}_2(\mathbb Z/n\mathbb Z)}
\newcommand{\SL}{\ensuremath{\mathrm{SL}}}
\newcommand{\SLZ}{\mathrm{SL}_2(\mathbb Z)}

\newcommand{\SLn}{\mathrm{SL}_2(\mathbb Z/n\mathbb Z)}

\newcommand{\Id}{\textnormal{I}}
\newcommand{\ns}{\textnormal{ns}}
\newcommand{\s}{\textnormal{s}}
\newcommand{\CM}{\textnormal{C}}

\newcommand{\lto}{\longrightarrow}
 
\newcommand{\vv}[2]{\left( \begin{smallmatrix} #1 \\#2 \end{smallmatrix} \right) } 
\newcommand{\diam}[1]{\langle #1 \rangle }

\numberwithin{equation}{section}

\newtheorem*{thm*}{Theorem} 
\newtheorem{thm}[equation]{Theorem}
\newtheorem*{prop*}{Proposition} 
\newtheorem{prop}[equation]{Proposition} 
\newtheorem*{lem*}{Lemma} 
\newtheorem{lem}[equation]{Lemma} 
\newtheorem*{cor*}{Corollary} 
\newtheorem{cor}[equation]{Corollary} 
\newtheorem{Algorithm}[equation]{Algorithm}

\newtheorem{Conjecture}[equation]{Conjecture}
\newtheoremstyle{named}{}{}{\itshape}{}{\bfseries}{.}{.5em}{#3}
\theoremstyle{named} 
\newtheorem*{namedtheorem}{Theorem} 
\theoremstyle{remark}
\newtheorem*{rem*}{Remark} 
 
\newtheorem{rem}[equation]{Remark} 
\newtheorem{ese}[equation]{Example} 
\newtheorem{Claim}[equation]{Claim} 
\theoremstyle{definition}
\newtheorem*{defi*}{Definition} 
\newtheorem{defi}[equation]{Definition} 

\usepackage{authblk}
\title{
Automorphisms of general modular curves
}
\author[1]{Valerio Dose}
\author[2]{Guido Maria Lido}
\author[3]{Pietro Mercuri}
\affil[1]{Sapienza University of Rome}
\affil[2]{Università di Roma Tor Vergata}
\affil[3]{Università degli Studi di Palermo}

\date{}

\begin{document}

\maketitle

\begin{abstract}
We prove that for a geometrically connected modular curve $X_H$ of prime level larger than $311$, all automorphisms are modular, i.e., they arise from the moduli interpretation of~$X_H$. The proof relies on a general criterion involving the gonality of~$X_H$ and the dimension of the CM part of the Jacobian of~$X_H$. For arbitrary level, we give algorithms to compute the modular automorphisms and, in many cases, to prove that there are no other automorphisms. Finally, we apply these algorithms to the curves in the database LMFDB.
\end{abstract}

\section{Introduction}

Given a smooth projective curve of genus $g \ge 2$ in characteristic~$0$, its automorphism group is  finite, as Schwarz showed in \cite{Schwarz}, and indeed has cardinality at most $84(g-1)$, by a theorem of Hurwitz \cite[Section~II.7, p.~424]{Hur}. In fact, for genus $g \ge 3$, a generic curve $C$ satisfies $|\Aut(C)|=1$; see, e.g., \cite{Moduli_cornalba, Moduli_sing, Moduli_Catanese} for the study, in moduli, of curves with non-trivial automorphisms. Moreover, considerable effort has been devoted to improving Hurwitz's bound and more generally to determining which groups arise as (subgroups of) automorphism groups of projective curves; see, e.g., \cite{Accola, Maclachlan, ConderAMS, BreuerBook}. 

In this article, we are concerned with the automorphism groups of modular curves that parameterize (as coarse moduli spaces) elliptic curves with level structure: given $H<\GLn$ for a positive integer~$n$, an $H$-structure is a basis of the $n$-torsion subgroup up to a change of basis in~$H$. 
Since we work with the coarse moduli spaces, we assume that $H$ contains $-\Id$. Moreover, we restrict ourselves to the case where $\det(H) = (\ZZ/n\ZZ)^\times $, so that $X_H$ is irreducible over~$\CC$. 
For such curves $X_H$, we have automorphisms that come from their moduli interpretation. For example, each matrix $m \in \GLn$ that normalizes~$H$ induces an automorphism of $X_H$ by a change of basis on the $n$-torsion of the elliptic curves.
More generally, one can define the notion of \emph{modular automorphism} by looking at the modular description of the complex points on~$X_H$, see Definition \ref{def:modaut}. For example, for the classical modular curves $X_0(n)$, the modular automorphisms were determined in the seminal paper \cite{AL70}, with a small gap which was later filled in a couple of different ways (see \cite{AS90} and \cite{Bar08}).

A natural question, then, is the following:
\begin{equation*}\label{eq: question}
\text{given a modular curve $X_H$, are all its automorphisms modular?}
\end{equation*}

This question has been extensively studied in specific families, and for each such family the answer turns out to be yes, 
up to finitely many exceptions. 
For classical Borel curves $X_0(n)$, this was proved in \cite{KM88}, with a few inaccuracies corrected in \cite{Har11} (the same papers, together with \cite{Elk90}, also solved the question completely for Borel curves).
For the curves associated to Cartan subgroups or their normalizers, this has been proven in \cite{DLM22} (and, for prime level, the same article together with \cite{Gon16} settled the question). 
Notice that the same phenomenon also occurs for Atkin-Lehner quotients $X_0^*(n)$ of $X_0(n)$, for square-free level~$n$, as proved in \cite{BG21} (these curves are not of the form $X_H$).
For the same families, see also the earlier works \cite{Ogg75aut}, \cite{Ogg77}, \cite{Dos16}, \cite{BH03}, and, particularly, Section~3 of the survey \cite{BarsSurvey}. 

Because of the above results, we call \emph{exceptional} the automorphisms that are not modular. Their existence often raises interest because of their connection to rational points.  
The best-known example is $X_0(37)$, which, as proved in \cite[Theorem~1]{Ogg74}, is hyperelliptic with the hyperelliptic involution being not modular and sending the two rational cusps to two non-CM, non-cuspidal rational points. Other examples can be found among the curves associated to non-split Cartan subgroups and their normalizers: the curve $X_\textnormal{ns}(11)$ has an ``exotic'' involution (see \cite{DFGS14}), which explains a symmetry in the Galois representations attached to the rational points of the elliptic curve $X_\ns^+(11)$, and $X_\textnormal{ns}^+(16)$ is hyperelliptic and bielliptic, with all the non-trivial automorphisms being non-modular, again giving rise to non-CM, non-cuspidal rational points (see \cite[Remark~A.1]{DLM22}).
More generally, in \cite[Sections~7 and~8]{derickx2026rationalpointsmodularcurves} Derickx, Hashimoto, Najman and Shnidman explain all the known rational points on modular curves by geometric means, that is, by taking lifts under abelian covers, or by applying algebraic maps (including automorphisms) or by other techniques called collinearity and fiber splicing. In their analysis, they exploit the non-modular automorphisms mentioned above, together with the non-modular automorphisms of the curves with LMFDB labels \href{https://beta.lmfdb.org/ModularCurve/Q/25.75.2.a.1}{25.75.2.a.1}, \href{https://beta.lmfdb.org/ModularCurve/Q/25.50.2.a.1}{25.50.2.a.1}, and \href{https://beta.lmfdb.org/ModularCurve/Q/45.54.2.c.1}{45.54.2.c.1}. Other interesting examples of modular curves with non-modular automorphisms (even though without known applications to rational points) are $X_0(63)$ (see \cite{Elk90}), $X_0(108)$ (see \cite{Har11}), $X_\textnormal{ns}^+(20)$ (see \cite[Remark ~A.1]{DLM22}), $X_{\textnormal{s}}^+(11)$ (see \cite{Gon16}), and  $X_{\Delta_3}(37)$ (the unique intermediate curve between $X_1(37)$ and $X_0(37)$ with the property that the map $X_{\Delta_3}(37) \to X_0(37)$ has degree~$3$, see \cite{Jeon_Intermediate_bielliptic}). 

In this work, we are concerned with general irreducible modular curves~$X_H$. Our main result is the following.
\begin{namedtheorem}[Theorem \ref{thm: prime level}]
All the automorphisms of $X_H$ are modular for each prime $p>311$ and for each each proper 
subgroup $H <  \GL_2(\ZZ/p\ZZ)$ containing $-\Id$ and satisfying $\det(H) = (\ZZ/p\ZZ)^\times$.
\end{namedtheorem}

\medskip
For composite levels, we give computational tools to study automorphisms of modular curves, namely an algorithm to compute the group of modular automorphisms (see Section \ref{sec:MAut}) and an easy-to-check criterion implying that all the automorphisms of a certain modular curve are modular (see Theorem \ref{cor:CShighgen} and Section \ref{sec:non-MAut}). 

Recently, motivated in part by Mazur's Program B, there has been growing interest in modular curves of composite level, and much of this work has relied on explicit computations, supported by the data collected in the beta version of the LMFDB. In this spirit, we also explicitly analyze modular curves of bounded level, obtaining the following results: 
\begin{itemize}
\item At least $63 \%$ of the modular curves of level $n \le 70$ and genus $g\ge2$ have only modular automorphisms, by checking the criterion of Theorem \ref{cor:CShighgen}; moreover, this holds for \emph{all} the $200\,170$ such curves of genus $g \ge 210$.
\item We found $139$ hyperelliptic modular curves admitting non-modular automorphisms.
\item We found $972$ modular curves with a ``large'' group of automorphisms, that is, they satisfy 
\[
|\mathrm{MAut}(X_H)| = |\mathrm{Aut}(X_H)|> \frac{1}{2}\max_{\operatorname{genus}(X)=\operatorname{genus}(X_H)}|\mathrm{Aut}(X)|,
\]
where $\mathrm{MAut}$ denotes the group of modular automorphisms.
\end{itemize}

The results in the literature and the explicit computations suggest the following.

\begin{Conjecture}
There exists an integer $g_0$ such that, for every geometrically connected modular curve $X_H$ of genus $g\ge g_0$,  all automorphisms of $X_H$ are modular.
\end{Conjecture}

The article is structured as follows. In Section \ref{sec: end jac}, we study properties of the endomorphisms of Jacobians of modular curves, which we then use in Section \ref{sec:theory_modularity} to prove our main theorem and sufficient conditions for automorphisms to be modular. In Section \ref{sec: algos} we detail our algorithms to study  automorphisms of modular curves, which we then use to make explicit computations, as described in Section \ref{sec: computational}. Finally, in the appendix, we list curves with many modular automorphisms or with exceptional automorphisms.

\subsection*{Acknowledgments}
We thank Maarten Derickx, Francesc Bars, Eran Assaf, John Voight and Jeremy Rouse for the fruitful discussions and Filip Najman for pointing out inaccuracies in a previous version of this article.

Lido is supported by the MIUR Excellence Department Projects MathMod@TOV awarded to the Department of Mathematics, University of Rome Tor Vergata,  and by the PRIN PNRR 2022 ``Mathematical Primitives for Post Quantum Digital Signatures''. 
Lido and Mercuri are members of the INdAM group GNSAGA.

\subsection*{Use of AI}
The results in Section \ref{sec: computational} rely on code written with the assistance of AI. The rest of the paper is entirely the authors' own work.

\subsection*{Notation and conventions}

Throughout the article, $n$ is a positive integer and $H$ is a subgroup of $\mathrm{GL}_2(\ZZ/n\ZZ)$ containing~$-\Id$.  

We denote by $X(n)$ the modular curve of full level~$n$, which is irreducible over $\QQ$ and has $\varphi(n)$ irreducible components over the algebraic closure $\ol{\QQ}$. We denote by $J(n)$ the Jacobian of $X(n)$ (the notion of Jacobian is well defined for all smooth, projective curves by \cite[Section~8.2, Theorem~3]{BLR90}; the geometric irreducibility is not necessary).
Furthermore, we denote by $X_H := X(n)/H$ the modular curve associated to~$H$, i.e. the coarse moduli space for elliptic curves with $H$-level structures  (see, e.g., \cite{DR73} for definition and properties of modular curves, and \cite[Section~2]{DLM22} or \cite{RSZ22} for a short recall); we denote  by $J_H$ the Jacobian of $X_H$.  
In particular, we denote by $X_1(n)$ the modular curve $X_H$ associated to $H=\left\{ \pm \smt ab01 :
a\in(\ZZ/n\ZZ)^\times,b\in\ZZ/n\ZZ\right\}$. We often assume that $\det(H) = (\ZZ/n\ZZ)^\times$, or equivalently that $X_H$ is geometrically connected.

We denote by $\HH:=\{\tau\in\CC:\mathrm{Im}(\tau)>0\}$ the complex upper half-plane and by $\overline{\HH}:=\HH\cup\QQ\cup\{\infty\}$ the extended complex upper half-plane. The group $\mathrm{GL}_2^+(\QQ) :=\{m\in\GL_2(\QQ):\det m>0\}$ acts on the left\footnote{Notice that in \cite{DR73} the action is a right action and that it corresponds to the above after a transposition.} on $\ol \HH$ by M\"obius transformations, i.e., by the formula
\[
\begin{pmatrix}a&b\\c&d\end{pmatrix}  \tau:=\frac{a\tau+b}{c\tau+d}.
\]
This action becomes faithful when we pass to the projectivization $\mathrm{PGL}_2^+(\QQ)$ of $\mathrm{GL}_2^+(\QQ)$. 
We denote by $\PP$ the natural map $\GL_2^+(\QQ)\to \mathrm{PGL}_2^+(\QQ)$.

When $X_H$ is geometrically connected, then $X_H(\CC)$ is biholomorphic to $\Gamma_H\backslash \ol \HH$, with
\[
\Gamma_{H}:=\{\gamma \in \mathrm{SL}_2(\ZZ): (\gamma  \bmod n) \text{ lies in }H\} 
\quad < \quad\mathrm{SL}_2(\ZZ)\quad < \quad\mathrm{GL}_2^+(\QQ).
\]
Indeed, if we associate to $\tau \in \HH$ the elliptic curve $E_\tau = \CC/(\ZZ+\ZZ\tau)$ together with the $H$-structure given by the isomorphism $\psi_\tau \colon (\ZZ/n\ZZ)^2 \to E_\tau[n]$ sending $\vv 10 \mapsto \frac 1n$ and $\vv 01 \mapsto - \frac \tau n$, 
we get a map $\HH\to X_H(\CC)$ that factors through $\Gamma_H\backslash \HH \hookrightarrow X_H(\CC)$, and extends to the isomorphism $\Gamma_H\backslash \ol \HH \cong  X_H(\CC)$.

Given a scheme $X$ over a field~$F$, we denote by $\Aut(X)$ the group of geometric automorphisms of $X$, that is the group of automorphisms of~$X_{\ol F}$ for $\ol F$ an algebraic closure of~$F$. Analogously, for $A/F$ an abelian variety, we denote by $\End(A)$ the ring of endomorphisms of~$A$, possibly defined over~$\ol F$.

\section{Endomorphisms of modular Jacobians}
\label{sec: end jac}
When the genus is at least~$2$, the group of automorphisms of a curve is embedded in the automorphisms of the Jacobian of the curve. In particular, to bound the field of definition of automorphisms it is enough to bound the field of definition of the automorphisms of the Jacobian. More generally, we study the action of the absolute Galois group $\Gal(\ol \QQ/\QQ)$ on the endomorphisms of the Jacobian of a modular curve.

We start from an explicit computation regarding specific automorphisms. For each subgroup $H$ of $\GLn$ and each $d$ in $(\ZZ/n\ZZ)^{\times}$, we have an automorphism $\langle d \rangle$ of $X_H$ acting on non-cuspidal points as follows: given an elliptic curve $E$ and a level structure $\phi\colon (\ZZ/n\ZZ)^2 \overset\sim\to E[n]$, we have   
\begin{equation}\label{diamond}
\langle d \rangle (E,\phi):= (E,d\phi),
\end{equation}
where $d\phi$ is the level structure obtained composing $\phi$ with the multiplication by $d$.
This definition extends the definition of the diamond operator given in \cite[Section~7.9]{DS05} on $X_1(n)$ to general level structures, and is compatible with the Eichler--Shimura relation (see \cite[Theorem 7.9 and Corollary 7.10]{Shi71intro} or for example \cite[Theorem~3.7]{CL23}). Furthermore, definition \eqref{diamond} differs from the one given in \cite[Appendix~A]{RSZ22} for full level structures.

Moreover, for each choice of a primitive $n$-th root of unity $\zeta$, we can lift the Fricke involution of $X_0(n)$ to $X_1(n)$ as follows:

\begin{align}\label{eq_def_W}
W\colon X_1(n) &\lto X_1(n),\\ \nonumber
(E,P) &\longmapsto \left(E/\langle P\rangle,Q+\langle P\rangle\right),
\end{align}
where $Q\in E[n]$ is a point such that the Weil pairing $\langle \cdot, \cdot \rangle$ on $E[n]$ satisfies
\[
\langle Q,P\rangle= \zeta.
\]
We call such a $W$ a Fricke automorphism and $\Gal(\ol\QQ/\QQ)$ acts on it as follows.

\begin{prop}\label{prop: Fricke and Galois}
Let $W$ be a Fricke automorphism on $X_1(n)$ as in \eqref{eq_def_W}. For each prime number $\ell \nmid n$ and for each Frobenius element $\sigma \in \Gal(\ol \QQ/\QQ)$ at $\ell$, we have 
\[
W^{\sigma} =  \langle \ell \rangle  W.
\]
\end{prop}
\begin{proof}
We prove that for each non-cuspidal point $(E,P)$ on $X_1(n)$ we  have $(\langle \ell \rangle W) (E,P) = W^\sigma (E,P) $. Let $\zeta$ be the root of unity in the definition \eqref{eq_def_W} of $W$, and let  $Q \in E[n]$ be a point with Weil pairing  $\langle Q,P\rangle=\zeta$. 
By construction, we have 
\begin{align*}
(\langle \ell \rangle W )(E,P) &= \left(E/\langle P\rangle,\ell Q+\langle P\rangle\right), \\
W^\sigma (E,P) &= W (E^{\sigma^{-1}},\sigma^{-1}(P))^\sigma= \left(E^{\sigma^{-1}}/\langle \sigma^{-1}(P)\rangle,Q'+\langle \sigma^{-1}(P)\rangle\right)^\sigma = \left(E/\langle P\rangle,\sigma(Q')+\langle P\rangle\right),
\end{align*}
where $Q'$ is a $n$-torsion point of $E^{\sigma^{-1}}$ such that $\langle Q',\sigma^{-1}(P)\rangle=\zeta$. By Galois-equivariance of the Weil pairing, we have $\langle \sigma(Q'), P\rangle = \sigma(\zeta) = \zeta^\ell$, which implies $\sigma(Q') \in \ell Q + \langle P\rangle $, hence the desired equality.
\end{proof}

We now generalize \cite[Theorem~2.3]{Rib75} to modular Jacobians, even when the Hecke algebra is not totally real.    
As in \cite{Rib75}, we want to avoid the CM abelian subvarieties of~$J_H$ (by ``CM abelian variety'' we mean an abelian variety that admits sufficiently many complex multiplications, see \cite[Definition~1.3.1.2 and Proposition~1.3.2.1]{CCO13}).

For every abelian variety~$A$ over a perfect field~$K$, we define the CM part $A^\CM$ of $A$ to be the maximal abelian subvariety of $A_{\ol K}$ that is isogenous to a product of simple CM abelian varieties. We analogously define the non-CM part $A^\mathrm{N}$ of~$A$ to be the maximal abelian subvariety of $A_{\ol K}$ that is isogenous to a product of simple non-CM abelian varieties. Both the CM part and the non-CM part of $A$ are also defined over~$K$, since by definition we have $(A^\CM)^\tau = A^\CM$ and $(A^\mathrm{N})^\tau = A^\mathrm{N}$ for each $\tau \in \Gal(\ol K/K)$. 
The dimension of the CM part and the dimension of the non-CM part are invariant under isogeny. Hence, by looking at the decomposition of $A_{\ol K}$ into simple factors, we see that $\dim (A^\CM) + \dim(A^\mathrm{N}) = \dim(A)$.

\begin{lem}\label{Ribetprop}
Let $A^\mathrm{N}$ be the non-CM part of the Jacobian of a modular curve $X_1(n)$ and for a prime $\ell\nmid n$ choose a Frobenius $\sigma\in \Gal(\ol\QQ/\QQ)$ relative to $\ell$. Then, for every abelian subvariety $A\subset A^\mathrm{N}$, with the inclusion defined over $\QQ$, and for every $u$ in $\End_{\ol\QQ} (A)$
we have
\[ 
u^{\sigma^{2}}=\langle\ell\rangle\,u\,\langle\ell^{-1}\rangle.
\]
\end{lem}
\begin{proof} 
Over $\QQ$, the irreducible factors of the Jacobian of $X_1(n)$ have the form $A_f$ for $f$ newform of level dividing~$n$. Since the non-CM part is defined over~$\QQ$ and the inclusion $A\subset A^\mathrm{N}$ is defined over~$\QQ$ too, we have an isogeny
\[
A  \sim \prod_{f\in S} A_f,
\]
for a certain multiset $S$ of newforms, i.e., a certain $A_f$ might appear multiple times.
Moreover, the above isogeny is defined over~$\QQ$ and is equivariant with respect to the diamond operators.

If $f$ has level $d$ dividing~$n$, then the Hecke algebra $\bT_d$ of $S_2(\Gamma_1(d))$ is a $\QQ$-algebra acting on $A_f$, i.e., there is a map $ \pi \colon \bT_d \to \End_\QQ(A_f) \otimes \QQ$. We can consider the quotient $\bT_f = \bT_d/\ker(\pi)$ which is a field of dimension $\dim_\QQ \bT_f = \dim A_f$ acting faithfully on $A_f$. Furthermore, its action on
\[
\Omega^1(A_f) = \bigoplus \CC f^\tau \subset S_2(\Gamma_1(d))
\]
is diagonalizable with eigenvectors the Galois-conjugates $ f^\tau$ of $f$, each one with a different system of eigenvalues, since the $q$-expansion of each $f^\tau$ is determined by the action of the Hecke algebra.

We consider the algebras
\[
\bT  := \prod_{f\in S} \bT_f  \qquad \subset \qquad \mathbb E := \End \left( A  \right) \otimes \QQ  =  \End \left( \prod_{f\in S} A_f  \right) \otimes \QQ,
\]
where the second is the full algebra of the geometric endomorphisms and the first  acts diagonally on $\prod_{f\in S} A_f$. 
To prove our statement we start by showing the following claim.
\begin{Claim}\label{Claim:Tcomm}
$\mathbb T$ is equal to its own commutant in $\mathbb E$.
\end{Claim}
One inclusion is the commutativity of the Hecke algebra $\bT$. For the other containment let  $c\in \mathbb E$ be an element commuting with $\mathbb T$. The action of $\bT$ on $\Omega^1(A)$ is diagonalizable and since each eigenvector in a basis $\mathcal B:=\{f_1, \ldots, f_r\}$ has a different system of eigenvalues, then such a basis is unique up to scaling each $f_i$. Since $c$ commutes with $\bT$, then the action of $c$ on $\Omega^1(A)$ is diagonalizable as well with the same basis of eigenvectors. 
In particular, for each $f\in S$, the morphism $c$ preserves each $ \bigoplus_{\tau \in \Gal(\ol\QQ/\QQ)} \CC \cdot f^\tau = \Omega^1(A_f)$, i.e., $c$  restricts to an endomorphism of each factor $A_f$.
In particular, it is enough to prove that $c$ acts on $A_f$ as an element of $\bT_f$ for each $f$.

Inside $\End(A_f)$ we consider the commutative ring $\mathbb T_f[c]$ which is isomorphic to $\mathbb T_f[x]/\mu(x)$ for some polynomial $\mu \in \mathbb T_f[x]$. Fix $\nu$ an irreducible factor of $\mu$ and let $B$ be the abelian subvariety of $A_f$ defined as the connected component of the identity in $\ker \nu(c)$. 
Since  $\bT_f[c]$ acts as  $\bT_f[x]/\nu(x)$ on $B$, if $B=A_f$ and $\deg \nu = 1$, then $c$ acts as an element of $\mathbb T_f$ as we want. In the remaining cases we derive a contradiction: if we had  $\dim B< \dim A_f$ or $\deg \nu \ge 2$, using that $\dim_\QQ \bT_f = \dim A_f$, we would have a field $L:=\bT_f[x]/\nu(x)$ inside $\End(B) \otimes \QQ$, of dimension $[L:\QQ]>\dim B$, implying that $B$ is of CM-type by \cite[Lemma 2.1]{Rib75} and this is impossible since $B$ is a subvariety of $A^\mathrm{N}$.

Now that we have proved our claim, we return to the main statement. First of all the claim implies that the center $Z$ of $\mathbb E$ is contained in $\bT$. Moreover we restate the result using the center: consider the map $\mathbb E \to \mathbb E$ given by $u \mapsto u^\sigma$, it is an automorphism of the semisimple algebra $\mathbb E$ which fixes $\bT$ (endomorphisms in the Hecke algebra are defined over $\QQ$), hence it fixes the center $Z$, consequently by  Skolem-Noether Theorem \cite[Ch. 8, Section 14, Theorem 4]{NicolasB}, it is an inner automorphism. In other words there exists $\alpha \in \mathbb E$ such that 
\begin{equation}\label{eq:Gal_conjugation}
u^\sigma = \alpha^{-1} u \alpha, \qquad \text{for each } u \in \mathbb E.
\end{equation}
In particular, $\langle\ell^{-1}\rangle u^{\sigma^{2}} \langle\ell\rangle = \langle\ell^{-1}\rangle \alpha^{-2} u \alpha^{2} \langle\ell\rangle$ and the main statement is equivalent to:
\begin{equation}\label{eq:commutation2}
\alpha^2 \langle\ell\rangle \text{ lies in the center of }\mathbb E.
\end{equation}
Since all the Hecke operators are defined over $\QQ$, then $\sigma$ acts trivially on $\mathbb T$, hence $\alpha$ commutes with $\mathbb T$, and by Claim \ref{Claim:Tcomm} $\alpha $ is an element of $\mathbb T$.
In particular for each eigenform $f_i$ in $\mathcal B$ 
there are complex numbers $\zeta_i$ and $\eta_i$ such that 
\begin{equation}\label{eta_zeta}
\langle \ell \rangle f_i = \zeta_i f_i, \qquad \alpha f_i = \eta_i f_i.
\end{equation}

Define $W$ to be the endomorphism of $\prod_{f \in S } A_f$ that, for each $f \in S$, denoting $d$ its level, acts on $A_f$ as a specific Fricke automorphism of $X_1(d)$. In particular, we choose the Fricke automorphism of $X_1(d)$ that acts on the complex points $\Gamma_1(d) \backslash \ol\HH$ as $\tau \mapsto -1/(d\tau)$: indeed, when $\tau$ varies in $\HH$, the Weil pairing of $P_\tau = \frac 1d $ and $Q_\tau = -\frac \tau d$ in $E_\tau = \CC/(\ZZ+\ZZ\tau)$ is constant, hence choosing $\zeta$ in \eqref{eq_def_W} to be equal to such a Weil pairing, the Fricke automorphism  sends $(E_\tau, \tfrac 1d)$ to $(E_{d \tau}, -\tau) \cong (E_{-1/(d\tau)}, \frac 1d)$. 
Moreover, by \cite[Theorem 4.6.15]{Miyake}, the action of this automorphism on the Jacobian of $X_1(d)$ leaves $A_f$ stable. By Proposition \ref{prop: Fricke and Galois} applied to the Jacobians of the curves $X_1(d)$, we have
\[
\alpha^{-1}W\alpha =  \langle \ell \rangle W.
\]
Again by \cite[Theorem 4.6.15]{Miyake}, $W$ sends an eigenform to a multiple of its complex conjugate: for each $f_i\in\mathcal B$ there is $\mu_i \in \CC^\times$ such that $Wf_i = \mu_i \ol{f_i}$. Combining with the above equality we get 
\[
\ol{\zeta_i} \mu_i \ol {f_i} = \langle \ell \rangle W  (f_i) = \alpha^{-1}W\alpha (f_i) = \alpha^{-1} (\mu_i\eta_i \ol {f_i}) = \tfrac{\eta_i}{\ol\eta_i}  \cdot \mu_i \ol {f_i},
\]
and hence
\begin{equation} \label{eq:alpha_bar}
\ol\eta_i = \eta_i \zeta_i. 
\end{equation}
To prove \eqref{eq:commutation2}, we prove that 
$Z[\alpha^2\langle\ell\rangle]$ is a trivial ring extension of the center $Z$ of $\mathbb E$. 
As we have seen, because of Claim \ref{Claim:Tcomm}, 
both $Z$ and  $Z[\alpha^2\langle\ell\rangle]$ are subrings of $\mathbb T$. In particular, since $\mathbb T$ is a finite $\QQ$-algebra without nilpotents, i.e., it is a product of number fields, hence the same remains true for its subrings $Z$ and  $Z[\alpha^2\langle\ell\rangle]$. 
In particular, as for number fields, the extension $Z\subset Z[\alpha^2\langle\ell\rangle]$ is non-trivial if and only if there are two morphisms $\sigma_1, \sigma_2 \colon Z[\alpha^2\langle\ell\rangle] \to \CC$ that agree on $Z$, but that send $\alpha^2\langle\ell\rangle$ to different complex numbers. By contradiction, we suppose that this is the case. 
As for number fields, we can extend these maps to maps $\sigma_1, \sigma_2 \colon \mathbb T \to \CC$. 
Being $\mathbb T$ a product of number fields, there are exactly $r = \dim_\QQ \mathbb T$ distinct morphisms $\mathbb T \to \CC$ and indeed we know $r$ distinct morphisms, namely the maps
$T \mapsto T(f_i)/f_i$, for $f_i$ the eigenforms in $\mathcal B$. In particular these are all the morphisms $\mathbb T \to \CC$ and we can suppose that $\sigma_1(T) = T(f_1)/f_1$ and $\sigma_2(T) = T(f_2)/f_2$. In particular using the eigenvalues $\eta_i$ and $\zeta_i$ in~\eqref{eta_zeta}, we have 
\begin{equation}\label{contra_eta_zeta}
\sigma_1(\alpha^2\langle\ell\rangle) =  \eta_1^2\zeta_1
\qquad\text{and}\qquad  \sigma_2(\alpha^2\langle\ell\rangle) = \eta_2^2\zeta_2.
\end{equation}
Anyway, we have seen before how $\eta_i$ behaves under complex conjugation, in particular, using \eqref{eq:alpha_bar}, we get
\[
\eta_i^2\zeta_i = \eta_i \ol{\eta_i} \in \RR_{\ge 0}.
\]
Moreover, there is a power $\sigma^e$ of $\sigma$ that  acts trivially, in fact there is a finite extension $L/\QQ$ where all the elements of $\End(A)$ are defined and therefore the action of $\Gal(\ol \QQ/\QQ)$ on $\End(A)$ factors through a finite quotient. We deduce that $\alpha^e$ lies in the center $Z$ and since $\langle\ell\rangle$ has finite order and commutes with~$\alpha$, up to enlarging~$e$, we have $(\alpha^2\langle\ell\rangle)^e \in Z$. In particular, since $\sigma_1$ and $\sigma_2$ agree on $Z$, we have
\[
(\eta_1^2\zeta_1)^e = \sigma_1\left( (\alpha^2\langle\ell\rangle)^e \right) =\sigma_2\left( (\alpha^2\langle\ell\rangle)^e \right) = (\eta_2^2\zeta_2)^e,
\]
and hence
\[
|\eta_2^2\zeta_2| = |\eta_1^2\zeta_1|.
\]
Since $\eta_1^2\zeta_1$ and $\eta_2^2\zeta_2$ are non-negative real numbers, they are equal to their respective absolute values, which together with Equation \eqref{contra_eta_zeta} implies
\[
\sigma_1(\alpha^2\langle\ell\rangle) = \sigma_2(\alpha^2\langle\ell\rangle),  
\]
leading to a contradiction.
\end{proof}

\begin{rem}
The last part of the proof, after Equation \eqref{eq:alpha_bar}, where we prove that $\alpha^2\langle\ell\rangle$ belongs to the center~$Z$, is modeled around the idea that we can write $\alpha^2\langle\ell\rangle = \beta^2$ for $\beta$ such that $Z[\beta]$ is a totally real extension of $Z$ and one could rephrase it in these terms, even though we find the above exposition more precise.
\end{rem}

Now we extend Lemma \ref{Ribetprop} to all modular curves $X_H$, passing through the Jacobian $J(n)$ of the full level modular curve $X(n)$. To describe the action of $\Gal(\ol\QQ/\QQ)$ on $\End(J(n))$, we look more closely at the geometric components of $X(n)$. Denote by $\mu_n^\times$  the set of  primitive $n$-th roots of unity. Then, the connected components of $X(n)_{\ol\QQ}$ are the curves $X_\zeta$, for $\zeta\in\mu_n^\times$, each parameterizing the tuples $(E,P,Q)$ such that $P$ and $Q$ are generators of $E[n]$ with Weil pairing $\langle P,Q\rangle=\zeta$.
Define, moreover 
\[
K = \left\{ \left( \begin{smallmatrix}
*&0\\0&1\end{smallmatrix}\right) \right\} < \GLn
\]
and consider the modular curve
\begin{equation}\label{eq:def_X'n}
X'(n) =X_K,
\end{equation}
which parametrizes tuples $(E,C,Q)$ where $E$ is an elliptic curve, $C$ is a cyclic subgroup of $E$ isomorphic to $\ZZ/n\ZZ$ and $Q$ is a point of order $n$ such that $E[n] =  C \oplus \langle Q \rangle$. 
Then, each curve $X_\zeta$ is isomorphic to $X'(n)$, simply by sending $(E,P,Q)$ to $(E,\diam P, Q)$, and we can bundle such isomorphisms
in the following map:
\begin{align} \label{eq_def_iota}
\iota\colon \bigsqcup_{\zeta\in\mu_n^\times}X'(n)_{\ol\QQ} & \overset{\sim}{\longrightarrow} \bigsqcup_{\zeta\in\mu_n^\times} X_\zeta =  X(n)_{\ol\QQ}, 
\\ \nonumber
\big((E,C,Q),\zeta\big)&\longmapsto (E,P,Q), 
\end{align}
with $P$ the unique point such that
\[
P\in C \qquad\text{and}\qquad \langle P,Q\rangle=\zeta.
\]
The map $\iota$ is not defined over $\QQ$, but only over its algebraic closure $\ol\QQ$. Denoting by $J'(n)$ the Jacobian of $X'(n)$, $\iota$ induces the isomorphism
\[
\iota = \iota_* \colon  
\mathrm{Jac}\left(\bigsqcup_{\zeta\in\mu_n^\times}X'(n)_{\ol\QQ} \right) = \prod_{\zeta \in \mu_n^\times}
J'(n)_{\ol\QQ}  
\overset\sim\lto J(n)_{\ol\QQ}.
\]
For each $a\in (\ZZ/n\ZZ)^\times$, we can define the automorphism
\[
S_a \colon 
\bigsqcup_{\zeta\in\mu_n^\times}X'(n)_{\ol\QQ} \lto
\bigsqcup_{\zeta\in\mu_n^\times}X'(n)_{\ol\QQ}, \qquad ((E,C,Q),\zeta) \longmapsto ((E,C,Q),\zeta^a) .
\]
When $S_a$ acts by pushforward on  $\mathrm{Jac}\left(\bigsqcup_{\zeta\in\mu_n^\times}X'(n)_{\ol\QQ} \right) = \prod_{\zeta \in \mu_n^\times} J'(n)$, we get a ``shift'' on tuples, that is  
\[
S_a ((x_\zeta)_{\zeta \in \mu_n^\times}) = S_{a,*} ((x_\zeta)_{\zeta \in \mu_n^\times}) = (x_{\zeta^{a^{-1}}})_{\zeta \in \mu_n^\times} .
\]

Using such a shift, we can describe the commutation between $\iota$ and diamond operators or elements in the absolute Galois group of $\QQ$.

\begin{lem}\label{iota_commutation}
Let $\iota$ be the isomorphism defined in Equation \eqref{eq_def_iota} and let $\ell$ be a prime number not dividing $n$.
\begin{enumerate}
\item 
Denoting $\diam \ell$ both the corresponding diamond operator on $J(n)$ and the automorphism of $\prod_{\zeta \in \mu_n^\times} J'(n)$ that acts as the diamond operator $\diam\ell$ on each factor $J'(n)$, we have 
$$\iota^{-1} \langle\ell\rangle \,\iota  = S_{\ell^2} \circ \langle\ell\rangle . $$
\item For each $u\in\End_{\ol\QQ}(J(n))$ and each choice of a  Frobenius  $\sigma\in \Gal(\ol\QQ/\QQ)$ at $\ell$, we have
$$
\iota^{-1} u^{\sigma^2} \iota =  S_{\ell^2} \circ \left( \iota^{-1} u\, \iota \right)^{\sigma^2} \circ S_{\ell^{-2}} .
$$
\end{enumerate}
\end{lem}
\begin{proof}
The first part follows observing that
\begin{align*}
(\iota^{-1} \langle\ell\rangle \iota) \big((E,C,Q),\zeta\big) &= (\iota^{-1} \langle\ell\rangle) (E,P,Q)=\iota^{-1} (E,\ell P,\ell Q)=\big((E,C,\ell Q),\zeta^{\ell^2}\big),
\end{align*}
where $P\in C$ is chosen such that $\langle P,Q\rangle=\zeta$.

For the second part, we observe that the statement is equivalent to
\[
u^{\sigma^2} = (\iota^{\sigma^2}  \circ S_{\ell^{-2}}\circ\iota^{-1} )^{-1} \circ u^{\sigma^2} \circ (\iota^{\sigma^2}  \circ S_{\ell^{-2}}\circ\iota^{-1}),
\]
and therefore, it is enough to prove that
\[
\iota^{\sigma^2}  \circ S_{\ell^{-2}}\circ\iota^{-1}=\mathrm{id}_{J(n)}.
\]
We prove the analogous equality on the curve.

We recall that for every non-cuspidal point $x=\big((E,C,Q),\zeta\big)$ the element $\sigma$ acts by
\[
x^\sigma=\big((E^\sigma,\sigma(C),\sigma(Q)),\zeta\big)
\]
and
\[
\iota^\sigma(x^\sigma)=\iota(x)^\sigma=(E^\sigma,\sigma(P),\sigma(Q)), \qquad\text{where }P\in C\text{ and }\langle P,Q\rangle=\zeta.
\]
Hence
\[
\iota^\sigma(x)=\iota(x^{\sigma^{-1}})^\sigma=(E^{\sigma^{-1}}, P',\sigma^{-1}(Q))^\sigma =(E,\sigma(P'),Q),
\]
where $P'\in \sigma^{-1}(C)$ and $\langle P',\sigma^{-1}(Q)\rangle=\zeta$ or, equivalently, $\sigma(P')\in C$ and $\langle \sigma(P'),Q\rangle=\zeta^\ell$.
Therefore $\sigma(P')=\ell P$ for $\langle P\rangle = C$ such that $\langle P,Q\rangle=\zeta$. Therefore, for $P$ and $Q$ such that $\langle P,Q\rangle=\zeta$, we have
\begin{align*}
\left(\iota^{\sigma^2}\circ S_{\ell^{-2}} \circ \iota^{-1}\right)(E,P,Q)&=\left(\iota^{\sigma^2}\circ S_{\ell^{-2}} \right)\big((E,\langle P\rangle,Q),\zeta\big) = \iota^{\sigma^2}\big((E,\langle P\rangle,Q),\zeta^{\ell^{-2}}\big)=(E,P',Q),
\end{align*}
where  $P'$ is chosen in $\langle P\rangle$ such that $\langle P',Q\rangle = (\zeta^{\ell^{-2}})^{\ell^2}=\zeta$. Hence $P'=P$ and we proved that
\[
\iota^{\sigma^2}\circ S_{\ell^{-2}} \circ \iota^{-1}=\mathrm{id}_{X(n)}.
\]
\end{proof}

Now we can extend Lemma \ref{Ribetprop} to all modular curves~$X_H$.
\begin{thm}\label{thm:J_H_commutation}
Let $H<\GLn$ be a subgroup satisfying $\det(H) = (\ZZ/n\ZZ)^\times$  and $-\Id \in H$. Let $J^\mathrm{N}_H$ be the non-CM part of $J_H$ and for a prime $\ell\nmid n$ choose a Frobenius $\sigma\in \Gal(\ol\QQ/\QQ)$ relative to~$\ell$. Then, for every abelian subvariety $A\subset J^\mathrm{N}_H$, with the inclusion defined over~$\QQ$, and for every $u$ in $\End_{\ol\QQ}(A)$ we have
\[ 
u^{\sigma^{2}}=\langle\ell\rangle \,u\,\langle\ell^{-1}\rangle.
\]
\end{thm}
\begin{proof}
It is enough to prove the statement for $u$ an endomorphism of the Jacobian $J(n)$ of $X(n)$ since $J_H$ is an isogeny factor of $J(n)$ and there is an inclusion 
\[
\End(J_H) \otimes \QQ \hookrightarrow \End(J(n)) \otimes \QQ 
\]
which preserves the meaning of $\langle\ell\rangle$ and $u \mapsto u^{\sigma^2}$: indeed, denoting $\pi$ the natural projection $X(n) \to X_H$, since $\pi_* \circ \pi^* = [\deg \pi]$, there is an isogeny defined over $\QQ$ and diamond-equivariant $J(n) \sim J_H \times (\ker \pi_*)^0$, where the superscript ${}^0$ denotes taking the connected component of the identity.

If we knew an isogeny defined over $\QQ$ and Hecke equivariant between $J(n)$ and a factor of the Jacobian of a certain $X_1(n')$, then by Lemma \ref{Ribetprop} we would have proven our goal, namely
\begin{equation} \label{eq:commutationJ(n)}
\langle\ell^{-1}\rangle\,u^{\sigma^{2}}\,\langle\ell\rangle=u   \qquad \text{in } \End(J(n)^\text{non-CM}) \otimes \QQ.
\end{equation}
We proceed with the same idea, but we use an isogeny which is not defined over $\QQ$: the map $\iota$ defined in Equation \eqref{eq_def_iota}. 
Since $\iota$ is an isomorphism, our goal Equation \eqref{eq:commutationJ(n)} is equivalent to the following
\begin{equation} \label{eq:commutation_iota}
(\iota^{-1} \langle\ell^{-1}\rangle\,\iota)\, 
(\iota^{-1}u^{\sigma^{2}}\iota)\,  
(\iota^{-1}\langle\ell\rangle\,\iota)= \iota^{-1}u\, \iota \qquad \text{in } \End\left( \prod_{\zeta \in \mu_n^\times} J'(n)  \right).
\end{equation}
Using Lemma \ref{iota_commutation}, and writing $\tilde u$ for $\iota^{-1}u\,\iota$, Equation \eqref{eq:commutation_iota} becomes
\[
 S_{\ell^{-2}} \circ \langle\ell^{-1}\rangle \circ  S_{\ell^2} \circ \tilde u^{\sigma^2} \circ S_{\ell^{-2}} \circ
  S_{\ell^2} \circ \langle\ell\rangle  = \tilde u, 
\]
where, we recall that the diamond operator $\diam\ell$ acts on $ \prod_{\zeta \in \mu_n^\times} J'(n)$ diagonally, that is it acts as the diamond operator $\diam \ell$ of $J'(n)$ on each factor.
Since shifts and diamond operators commute, the above equality is equivalent to the following 
\begin{equation}\label{commutation1000}
\langle\ell^{-1}\rangle  \circ  \tilde u^{\sigma^2} \circ \langle\ell\rangle  = \tilde u. 
\end{equation}
We can identify the endomorphisms of $\prod_{\zeta \in \mu_n^\times} J'(n)$ with matrices whose entries are in $\End(J'(n))$, e.g., we can write $\tilde u=(u_{ij})$, with $u_{ij}\colon J'(n) \to J'(n)$, for $i,j\in \mu_n^\times$. Equation \eqref{commutation1000} above becomes 
\[
\langle\ell^{-1}\rangle\,u_{ij}^{\sigma^{2}}\,\langle\ell\rangle=u_{ij}.
\]
In other words, we reduced Equation \eqref{eq:commutationJ(n)} to the same statement for $J'(n)$. 
This can be reduced to Lemma \ref{Ribetprop} as follows: the map $X_1(n^2) \to X'(n)$ sending a non-cuspidal point $(E,P)$ to the point $(E/\langle nP \rangle, E[n]/\langle nP \rangle, P)$ is defined over $\QQ$ and diamond-equivariant and induces, by pullback, an  isogeny between $J'(n)$ and an isogeny factor of $J_1(n^2)$, and consequently  
an inclusion $\End(J'(n)) \otimes \QQ \hookrightarrow \End(J_1(n^2)) \otimes \QQ$ which preserves the meaning of $\langle\ell\rangle$ and $u \mapsto u^{\sigma^2}$.
\end{proof}

Theorem \ref{thm:J_H_commutation}, together with \cite{Rib75} and \cite{CMLarioGon}, gives us a bound on the field of definition of the endomorphisms of factors of modular Jacobians. 
By \cite{RSZ22}, the Jacobian $J_H$ is isogenous over~$\QQ$ to a product of simple factors $A_f$ for $f$ varying over the Galois orbits of newforms of $S_2(\Gamma_1(d))$ for $d\mid n^2$. Since the factors $A_f$ are simple, as we explained in the discussion before Lemma \ref{Ribetprop} above, each $A_f$ is either completely non-CM or completely CM. 
By \cite[Section~4]{Shi71CM} and \cite{Shi72}, we know that $A_f$ is CM if and only if $f$ is a modular form associated to a Gr\"ossencharacter of an imaginary quadratic field~$K$.

\begin{cor}\label{cor:field_subabelian_varieties_J_H}
Let $H<\GLn$ be a subgroup satisfying $\det(H) = (\ZZ/n\ZZ)^\times$  and $-\Id \in H$, and let $A \sim \prod_{f\in S} A_f^{m_f} $ be an abelian subvariety of a modular abelian variety $J_H$, with $S$ a set of newforms of level dividing $n^2$ and $m_f\in \ZZ$. 
\begin{enumerate}
\item \label{item:compositum_quad} If all the newforms $f\in S$ have trivial character and do not have CM, then all the $\ol \QQ$-endomorphisms of $A$ are defined over a compositum of quadratic extensions of~$\QQ$.
\item\label{item:field_zetan}  If all the newforms $f\in S$ do not have CM, then all the $\ol \QQ$-endomorphisms of $A$ are defined over a compositum of quadratic extensions of the cyclotomic field $\QQ(\zeta_n)$.
\item \label{item:field_withCM} If all the newforms $f\in S$ have CM by the quadratic imaginary field $K$, then all the $\ol \QQ$-endomorphisms of $A$ are defined over  
the ray class field of $K$ associated to the ideal $n^2 \calO_K$. 
\end{enumerate}
\end{cor}
\begin{proof} 
Part~\ref{item:compositum_quad} is a consequence of \cite[Theorem~2.3]{Rib75}: since the character is trivial, the Hecke algebra acts on each $A_f$ as a totally real field of dimension $\dim A_f$, hence the hypotheses of Ribet's theorem apply.

Part~\ref{item:field_zetan} follows from Theorem \ref{thm:J_H_commutation}. Let $L$ be the minimal field such that all endomorphisms of $A$ are defined over~$L$. Then $L$ is a Galois extension of~$\QQ$, and it is enough to prove that $L(\zeta_n)/\QQ(\zeta_n)$ is a compositum of quadratic extensions, or equivalently that for all elements $\tau \in \Gal(L(\zeta_n)/\QQ(\zeta_n))$ we have $\tau^2 = \Id$. Equivalently, by the definition of $L$, it is enough to prove that, for all endomorphisms $u$ of $A$, we have
\begin{equation}\label{eq:wish_field_of_definition}
u^{\tau^2} = u.
\end{equation}
Let $\mathcal O$ be the ring of integers of~$\QQ(\zeta_n)$. By the Chebotarev density theorem, there exists a prime $\mathcal L$ of~$\mathcal O$, not dividing~$n$, whose Frobenius element is, up to conjugation, $\tau$. Moreover, denoting by $\ell\in\ZZ$ the prime below $\cal L$ and by $k$ the inertia degree, we can write $\tau = \sigma^k$, with $\sigma \in \Gal(L(\zeta_n)/\QQ)$ a Frobenius at~$\ell$. Now, since $A$ is isogenous to an abelian subvariety of $J_H^{\textnormal N}$ defined over~$\QQ$, we can apply Theorem \ref{thm:J_H_commutation}, which gives
\[
u^{\tau^2} = u^{\sigma^{2k}} = \langle \ell^k\rangle\, u \,\langle \ell^{-k}\rangle.
\]
Now \eqref{eq:wish_field_of_definition} follows from the fact that the inertia degree $k$ of $\mathcal L$ is equal to the multiplicative order of $\ell$ modulo~$n$, hence $\langle \ell^k\rangle = \langle 1 \rangle$ is the identity.

Part~\ref{item:field_withCM} is a consequence of \cite{CMLarioGon}.
Let $\calO_K$ be the maximal order of $K$ and let $F$ be the ray class field of $K$ associated to the ideal $n^2 \calO_K$. 
Each newform $f\in S$ is the modular form associated to a primitive Gr\"os\-sen\-cha\-rac\-ter of~$K$, with a certain conductor $\mathfrak c_f$. By \cite{CMLarioGon}, particularly  Proposition~3.1 and the last four lines of the proof, there is an isogeny $A_f\sim E_f^{\dim A_f}$ which is defined over the ray class field $H_{\mathfrak c_f}$ of~$K$ associated to~$\mathfrak c_f$, and $E_f$ is an elliptic curve, defined over the same field, with the following properties:
\begin{enumerate}
\item $\End(E_f) =\calO_K$;
\item  \label{item : Frob Ef} there is a Gr\"ossencharacter $\psi'$ of conductor either $\mathfrak c_f$ or $\ol{\mathfrak c_f}$  such that, for each prime $
\frak p$ of $K$ and each prime $\mathfrak P$ of $H_{\mathfrak c_f}$ over $\frak p$, the Frobenius of $(E_f \bmod {\frak P})$ has characteristic polynomial
\[
\Big(x-\psi'(\frak p^{k_f})\Big)\Big(x-\ol{\psi'(\frak p^{k_f})}\Big),
\]
where $k_f$ is the inertia degree of $\frak P/ \frak p$. 
\end{enumerate}
Notice that each field $H_{\mathfrak c_f}$ is contained in $F$ since, by \cite[p.~168]{Shi72}, the level of $f$ is $-\mathrm{disc}(K)\mathrm{Norm}(\mathfrak c_f)$, and consequently $\mathfrak c_f$ divides $n^2\calO_K$. Moreover, in the context of Item \ref{item : Frob Ef}, if $\mathcal P$ is a prime of $F$ extending~$\mathfrak P$ and $k$ is the inertia degree of~$\mathcal P/ \frak p$, then the ideal ${\frak p}^{k}$ is principal, generated by an element $\alpha_{\frak p}$ that is congruent to $1$ modulo~$n^2\calO_K$. So $\alpha_{\frak p}$ is congruent to $1$ modulo both ${\frak c}_f$ and its conjugate. In particular, we have $\psi'(\mathfrak p^{k}) = \alpha_{\frak p}$, which implies that the Frobenius of $(E_f \bmod {\mathcal P})$ has the same characteristic polynomial as $\alpha_{\frak p}$, which does not depend on~$f$. Now, varying $\cal P$, by Faltings criterion (e.g. \cite[\S 2, Corollary 2]{cornell1986arithmetic}), this implies that all the elliptic curves $E_f$ are isogenous over $F$. In particular, fixing one of them and calling it $E$, we have an isogeny $A\sim E^{ \dim A}$ also defined over~$F$. Since $K$ is contained in~$F$, all the endomorphisms of $E$ are defined over~$F$, which gives our statement.
\end{proof}

\section{Modularity of the automorphism group}
\label{sec:theory_modularity}

The aim of this section is to prove that, under certain hypotheses, all the automorphisms of a modular curve are modular in the following sense.
\begin{defi}\label{def:modaut}
Let $H$ be a subgroup of $\GLn$. An automorphism $u$ of the Riemann surface $\Gamma_H\backslash \overline{\HH}$ is \emph{modular} if it is described by a M\"obius transformation associated to a matrix $m \in \mathrm{PGL}^+_2(\QQ)$, namely for each $\tau \in \ol\HH$ we have $u(\Gamma_H \tau) = \Gamma_H m\tau$.

When $\det(H)= (\ZZ/n\ZZ)^\times$, an automorphism of $X_{H}$ is modular if its action on $X_H(\CC) = \Gamma_H\backslash \overline{\HH}$ is modular.
\end{defi}

We follow a strategy similar to \cite{DLM22} and use the following criterion.

\begin{lem}\label{prop:main}
Let $H<\GLn$ be a subgroup satisfying $\det(H) = (\ZZ/n\ZZ)^\times$  and $-\Id \in H$.
Suppose that there is a prime $\ell \nmid n$  with the following property: 
for every $u\in \Aut(X_H)$ there exists another automorphism $u' \in \Aut(X_H)$ such that the Hecke correspondence $T_\ell$ satisfies
\begin{equation}\label{eq:commuu'}
T_\ell u = u' T_\ell   \quad \text{in } \mathrm{Div}(X_H).  
\end{equation}
Then, every automorphism of $X_H$ is modular.
\end{lem}
\begin{proof}

In order to prove that every automorphism $u$ is modular, it is enough to show that $u$ preserves the set of cusps and the set of branch points of the map $\HH \to Y_H(\CC)$.
Indeed, when this occurs, applying \cite[Proposition~3.1]{Dos16}, we deduce that $u$ is induced by an automorphism $v \colon \HH \to \HH$. We know that $\Aut(\HH)= \tx{PGL}_2^+(\RR)$, hence $v$ is a M\"obius transformation given by a matrix $m \in \tx{PGL}_2^+(\RR)$. Since $v$ factors through the quotient, $m$ belongs to the normalizer of $\Gamma_H$ in $\tx{PGL}_2^+(\RR)$ and actually, by \cite[Remark~5.2]{DLM22}, to the normalizer of $\Gamma_H$ in $\tx{PGL}_2^+(\QQ)$. Then, by definition, $u$ is modular.

We now prove that all automorphisms preserve the set of cusps and the set of elliptic points, which is the set of branch points of the map $\HH \to Y_H(\CC)$. 

When $\ell \ge 5$, \cite[Theorem~2.4]{DLM22} allows us to distinguish those points, depending on their image under $T_\ell$: if $P$ is a cusp, then $T_\ell (P)$ contains a point with multiplicity at least~$4$; if $P$ is an elliptic point of order~$3$, then $T_\ell (P)$ contains a point with multiplicity~$3$; if $P$ is an elliptic point of order~$2$, then $T_\ell (P)$ contains two distinct points with multiplicity~$2$; otherwise, $T_\ell (P)$ has a different ``shape'', i.e., $T_\ell (P)$ is either made of $\ell+1$ points of multiplicity~$1$, or $\ell-1$ points of multiplicity~$1$ and one of multiplicity~$2$. 
Moreover, Equation \eqref{eq:commuu'} tells us that, for $u$ an automorphism of $X_H$, the divisors $T_\ell (P)$ and $T_{\ell}u(P)$ have the same ``shape'', implying that $u$ preserves both the set of cusps and the set of elliptic points.

We now treat the case $\ell=3$. We start by proving that the set of cusps is stable under~$u$. Notice that iterating~\eqref{eq:commuu'}, we deduce that, for every power $T_\ell^k$ of $T_\ell$, we have that $T^k_\ell(P)$ and $T^k_\ell u(P)$ have the same shape. The set of cusps is stable under Hecke operators, hence, if $C$ is a cusp, the cardinality of the support of $T^k_\ell(C)$ is bounded when $k$ tends to $\infty$. If, by contradiction, we have $u(C) = (E,\phi)$ is a non-cuspidal point, then also the support of $T_\ell^k(E,\phi)$ is bounded when $k$ tends to $\infty$, and \emph{a fortiori}, the support of $T_\ell^k(E)$ in $X(1)$.  But this is absurd since the cardinality of the support of $T_\ell^k(E)$ tends to infinity with~$k$: if $E$ is non-CM, its support contains all the points of $T_{\ell^k}(E)$, which are distinct; moreover, it also contains the points of $T_{\ell^{k-2}}(E)$, $T_{\ell^{k-4}}(E)$, $\ldots$, thus, if $E$ is a CM point, for $k$ sufficiently large the support of $T_\ell^k(E)$ contains points with CM by arbitrarily many distinct imaginary quadratic orders.
It remains to show that the elliptic points are stable under $u$. Since $u$ and $u^{-1}$ send cusps to cusps, we know at least that $u$ sends branch points to non-cuspidal points. Then, by Propositions~$2.7$ and $2.9$ in \cite{DLM22}, branch points $P$ of order~$2$ are characterized by the fact that $T_3(P)=2P_1+2P_2$ and branch points $P$ of order~$3$ are characterized by the fact that $T_3(P)=3P_1+P_2$, which gives what we want.

Finally, we treat the case $\ell=2$. The same argument used in the case $\ell=3$ proves that the set of cusps is stable. Propositions~$2.7$ and $2.9$ in \cite{DLM22} imply that the elliptic points $P$ of order $3$ are stable, since they are characterized by the fact that $T_2(P)=3P'$. 
It remains to show that the elliptic points of order $2$ are characterized. We use the notation $E_\tau$ for the elliptic curve $\CC/\ZZ+\ZZ\tau$. Then an elliptic point of order $2$ is of the form $(E_i,\phi)$ and Proposition~2.9 of \cite{DLM22} tells us that $T_2(E_i,\phi)=2P_1+P_2$ with $P_1\neq P_2$, and in particular, 
\[
T_2(E_i,\phi) = 2(E_{2i}, \phi')+ (E_{i}, \phi'').
\]
Since the cusps are preserved by automorphisms, then $u(E_i,\phi)$ is not a cusp, i.e., we can write $u(E_i,\phi) = (E,\psi)$ for an elliptic curve~$E$, and  Equation~\eqref{eq:commuu'} gives
\begin{equation}\label{eq:commuEi}
2 u'(E_{2i}, \phi')+ u'(E_{i}, \phi'')  = u'(T_2(E_i,\phi)) = T_2 (u(E_i,\phi)) = T_2(E,\psi)= 2(E',\psi') + (E'',\psi'').
\end{equation}
Now, as in \cite[Lemma~2.6]{DLM22}, this means that $E'$ has a cyclic endomorphism of degree $4$, which implies that $\End(E')$ is one of the following quadratic orders:
\begin{equation}\label{eq:orders}
\ZZ[2i], \quad \ZZ[\sqrt{-3}],  \quad  \ZZ[(1+\sqrt{-7})/2],  \quad \ZZ[(1+\sqrt{-15})/2].
\end{equation}
In each case, we can list the possible curves $E'$ with that endomorphism ring. Moreover, the fact that $(E',\psi')$ appears twice in $T_2(E,\psi)$ implies that there are two $2$-isogenies $\alpha, \beta\colon E \to E'$ having different kernels and with the additional property that 
\begin{equation}\label{eq<tech}
(E', \alpha\circ \psi) = (E', \beta\circ \psi) = (E', \psi').    
\end{equation}
For each $E'$, this also gives finitely many possibilities for~$E$.

First suppose that $\End(E')$ is one of the last two orders in~\eqref{eq:orders}. Then, by direct computation, all possible pairs $(E,E')$ satisfy $\End(E) = \End(E')$. We also can see that $T_2(E', \psi')$ contains a double point: the $2$-isogenies $\alpha^\vee, \beta^\vee\colon E' \to E$ have different kernels (here we use that $\Aut(E) = \Aut(E') = \{\pm 1\}$), hence $T_2(E', \psi')$ contains both $(E,\alpha^\vee \circ\psi')$ and $(E,\beta^\vee \circ\psi')$; these two points are the same point in $X_H$ since, identifying $\alpha$, $\alpha^\vee$, $\beta$, and $\beta^\vee$ with their actions on the $n$-th torsion subgroup, we can rephrase \eqref{eq<tech} as follows
\begin{equation}\label{eq<tech2}
\psi^{-1}\circ \alpha^{-1} \circ \beta \circ \psi = \tfrac 12 \psi^{-1}\circ \alpha^{\vee} \circ \beta \circ \psi  \quad \in H,
\end{equation}
which implies
\[
(\psi')^{-1} \circ (\beta^{\vee})^{-1} \circ \alpha^\vee \circ\psi' = \psi^{-1} \circ \alpha^{-1} \circ (\beta^{\vee})^{-1} \circ \alpha^\vee \circ \alpha\circ\psi = \psi^{-1} \circ \alpha^{-1} \circ (\beta/2) \circ 2 \circ\psi  = \psi^{-1}\circ \alpha^{-1} \circ \beta \circ \psi \quad \in H.
\]
Now Equation \eqref{eq:commuEi} implies $u'(E_{2i}, \phi') = (E', \psi')$, but this is absurd since $T_2(E', \psi')$ has a double point, while $T_2(E_{2i}, \phi')$ does not (indeed, $T_2(E_{2i})$ already over $X(1)$ does not have double points). 

Now suppose $\End(E') = \ZZ[\sqrt{-3}]$. Then, there is only one choice for $E'$, up to isomorphism, that is $E' = E_{\sqrt{-3}}$, and consequently $E = E_\rho$, with $\rho =(1+\sqrt{-3})/2$ (it is the only curve, up to isomorphism, which has two $2$-isogenies $\alpha, \beta\colon E \to E'$ having different kernels). To understand $T_2(E,\psi)$, we look at the cyclic subgroups $C_1$, $C_2$, and  $C_3$ of $E[2]$: up to reordering, we have $C_2 = \rho C_1$ (here $\rho$ is identified with the corresponding automorphism of $E$) and $C_3 = \rho^2C_1$ and $E/C_i \cong E'$ for $i\in \{1,2,3\}$, so that we can write 
\[
T_2(E,\psi) = (E', \psi_1) + (E', \psi_2) + (E', \psi_3).
\]
with the properties $\psi_2 = \rho\circ \psi_1$ and $\psi_3 = \rho^2\circ \psi_1$ (here $\rho$  actually denotes $\frac{1}{2}(2\rho)|_{E'[n]}$, with $2\rho$ identified with the corresponding endomorphism of $E'$, and analogously for $\rho^2$). If two of the points  in $T_2(E,\psi)$ are equal, using that $E'$ has trivial automorphisms, we get that either $\psi_1^{-1}\circ \rho\circ \psi_1$ or  $\psi_1^{-1}\circ \rho^2\circ \psi_1$ lie in~$H$. But since one is the square of the other, if one lies in~$H$, the other lies in $H$ as well. This implies that $T_2(E,\psi) = T_2(u(E_i,\phi))$ actually has the shape $3P_1$ which contradicts Equation \eqref{eq:commuEi}, since $T_2(E_i, \phi)$ is not made of one point with multiplicity~$3$.

Having excluded all the other possibilities in~\eqref{eq:orders}, we deduce that $\End(E') = \ZZ[2i]$. This implies that $E' \cong E_{2i}$ and $E = E_{i}$. But now, the fact that $T_2(E,\psi) = T_2(E_i,\psi)$ has a double point implies, by \cite[Proposition~2.9]{DLM22}, that $(E,\psi)$ is an elliptic point (of order $2$) as desired.
\end{proof}

For future reference, we recall the following result. 
\begin{lem}\label{lem:injaut}
Let $\ell$ be a prime number. 
If $A/\QQ$ is an abelian variety with good reduction at $\ell$, then the reduction map 
$\End(A_{\overline\QQ}) \to \End(A_{\overline{\FF}_\ell})$ is injective.
Moreover, if $X/\QQ$ is a curve of genus $g\ge 2$ and with good reduction at $\ell$, then the reduction map 
$\Aut(X_{\overline\QQ}) \to \Aut(X_{\overline{\FF}_\ell})$ is injective.
\end{lem}
\begin{proof}
For the first part, it is enough to choose a prime $p\neq\ell$ and look at the  action of endomorphisms  on the Tate module $T_p A = \varprojlim A[p^n](\ol\QQ)$, which is isomorphic to $T_p A_{\ol \FF_\ell}$ since the reduction $A[p^n](\ol\QQ) \to A[p^n](\ol\FF_\ell)$ is an isomorphism: an endomorphism of $A$ is zero on $A_{\ol \FF_\ell}$ if and only if it acts as zero on $T_p A_{\ol \FF_\ell}$, which implies that it acts as zero on $T_p A$, hence that the endomorphism is zero. 

For the second part, denoting by $J$ the Jacobian of $X$, it is enough to prove that the composition 
\[
\Aut(X_{\overline\QQ}) \to \Aut(X_{\overline{\FF}_\ell})  \to \Aut( J_{\overline{\FF}_\ell})
\]
is injective or, equivalently, that it is injective the  composition 
\[
\Aut(X_{\overline\QQ}) \to \Aut( J_{\overline{\QQ}})  \to \Aut( J_{\overline{\FF}_\ell}),
\]
where an automorphism of the curve induces an automorphism of the Jacobian, as algebraic group, by pushforward of divisors.
The map $\Aut( J_{\overline{\QQ}})  \to \Aut( J_{\overline{\FF}_\ell})$ is injective by the first part of the statement, hence it is enough to show that the map $\Aut(X_{\overline\QQ}) \to \Aut( J_{\overline{\QQ}})$ is injective. 
By contradiction, let $u \ne \mathrm{id}$ be an automorphism in the kernel. Then, for each pair of points $x$ and $y$ on the curve, the point $u_*(x-y) = u(x)-u(y)$ in the Jacobian is equal to $x-y$, implying that the divisor $u(x) +y -x -u(y)$ is the divisor of a rational function $f_{x,y}\colon X \to \PP^1$. Fix $x$ such that $x\neq u(x)$. Then, for all but finitely many choices of  $y$, the function $f_{x,y}$ has degree $2$. By looking  at  the action of function fields, we deduce that $f_{x,y}$ is Galois, hence there exists an automorphism $\iota_{x,y} \colon X \to X$ that permutes nontrivially the fibers of $f_{x,y}$, and therefore sends $x$ to $u(y)$. In particular, when $y$ varies, also $u(y)$ takes infinitely many values, implying that there are infinitely many different automorphisms of $X$, contradicting Hurwitz's result in \cite{Hur} on the finiteness of $\Aut(X)$.
\end{proof}

Our main strategy, in order to prove that the automorphisms on a certain modular curves are all modular, is to prove that Equation \eqref{eq:commuu'} is satisfied. We proceed in two steps. 
The following result is the first step, showing that, under suitable hypotheses, \eqref{eq:commuu'} holds in the Jacobian of the curve.

\begin{lem}\label{th:commrel}
Let $H<\GLn$ be a subgroup satisfying $\det(H) = (\ZZ/n\ZZ)^\times$  and $-\Id \in H$. 
Suppose that $J_H = \mathrm{Jac}(X_H)$ has ``small CM part" $J_H^\mathrm{C}$, namely suppose that
\[
\dim J_H^\mathrm{C} \le \tfrac 12 \dim J_H - 1.
\]
Then, for each prime $\ell \nmid n$, for each $u \in \Aut(X_H)$, and for each $\sigma \in \mathrm{Gal}(\overline{\QQ}/\QQ)$ Frobenius element at $\ell$, we have
\begin{equation}\label{eq:main_commutation}
T_\ell u=u^{\sigma} T_\ell \quad \text{in } \End_{\ol{\QQ}}(J_H),
\end{equation}
where $T_\ell$ is the usual $\ell$-th Hecke operator on $J_H$, and $u$ is identified with $u_*\colon J_H\to J_H$.
\end{lem}
\begin{proof}
By Lemma \ref{lem:injaut}, it is enough to prove that Equation \eqref{eq:main_commutation} holds modulo~$\ell$. 
By Eichler--Shimura (see \cite[Theorem~7.9 and Corollary~7.10]{Shi71intro} or \cite[Theorem~3.7]{CL23}), and by the properties of the Frobenius morphism $\mathrm{Frob_\ell}$, we have 
\[ 
T_\ell u \equiv (\mathrm{Frob}_\ell)_* u+ \langle\ell\rangle (\mathrm{Frob_\ell})^* u \equiv u^\sigma (\mathrm{Frob_\ell})_*+ \langle\ell\rangle u^{\sigma^{-1}} (\mathrm{Frob_\ell})^{*} \pmod\ell.
\]
Moreover
\[
u^\sigma T_\ell\equiv u^\sigma (\mathrm{Frob_\ell})_*+\ u^\sigma\ \langle\ell\rangle(\mathrm{Frob_\ell})^{*} \pmod\ell.
\]
In particular, if $\langle\ell\rangle\ u^{\sigma^{-1}} \equiv  u^\sigma\ \langle\ell\rangle\pmod\ell$, then $u^\sigma T_\ell \equiv  T_\ell u\pmod\ell$, which implies our proposition. 
Therefore, it is enough to prove that
\begin{equation} \label{eq:commutation}
u^{\sigma^{2}}=\langle\ell\rangle\, u\,\langle\ell^{-1}\rangle.
\end{equation}
Actually, it is enough to prove the above relation for $u$ an endomorphism of a piece of $J_H$ big enough: the arguments in the proof \cite[Lemma~4.1]{DLM22} imply that two automorphisms $u_1, u_2$ of a curve are equal if and only if their actions on the Jacobian $J_H \sim A_1 \times A_2$ agree on $A_1$ and $\dim A_1 > \dim A_2+1$. 
In particular, we have the decomposition $J_H \sim J_H^\CM \times J_H^\text{N}$ into CM part and non-CM part. The assumption $\dim J_H^\mathrm{C} \le \tfrac 12 \dim J_H - 1$, implies exactly that $\dim J_H^\text{N} > \dim J_H^\mathrm{C}+1$, hence it suffices to prove that \eqref{eq:commutation} holds for every endomorphism $u$ of $J_H^\text{N}$. Notice that the decomposition $J_H \sim J_H^\CM \times J_H^\text{N}$ is defined over $\QQ$, hence both pieces are stable under $\sigma$. Moreover, since $J_H^\CM$ and $J_H^\text{N}$ have non-isomorphic simple pieces, then both $J_H^\CM$ and $J_H^\text{N}$ are stable under every endomorphism of $J_H$. In particular, they are stable under $\langle \ell \rangle$ and it is enough to prove the following
\begin{equation*} \label{eq:commutation1.5}
u^{\sigma^{2}}=\langle\ell\rangle\, u \,\langle\ell^{-1}\rangle  \qquad \text{in } \End(J_H^\text{N}) \otimes \QQ.
\end{equation*}
But this is true by Theorem \ref{thm:J_H_commutation}.
\end{proof}

We now prove that, if we keep the assumption of Lemma \ref{th:commrel}, and we add an hypothesis on the gonality of the curve, Equation \eqref{eq:commuu'}  holds, and consequently all the automorphisms of $X_H$ are modular.

\begin{thm}\label{cor:CShighgen} \label{th:highgenus}
Let $H<\GLn$ be a subgroup satisfying $\det(H) = (\ZZ/n\ZZ)^\times$  and $-\Id \in H$, 
and let $J_H^\mathrm{C}$ be the CM part of $J_H = \mathrm{Jac}(X_H)$. If the following two conditions are satisfied
\begin{enumerate}
\item \label{condition_CM} $\dim J_H^\mathrm{C} \le \tfrac 12 \dim J_H - 1$,
\item  \label{condition_gonality} $\tx{gonality}_{\ol{\QQ}}(X_H)\ge 2\ell+2$ for a prime $\ell\nmid n$,  
\end{enumerate}
then every automorphism of $X_H$ is modular.
\end{thm}
\begin{proof}
Notice that  $u^\sigma T_\ell$ and $T_\ell u$ give correspondences in $(X_{H} \times X_{H})_{\ol\QQ}$ which are irreducible and $1$-dimensional. Such correspondences either coincide or have finite intersection. 
If they coincide, then Lemma \ref{prop:main} implies our statement. 
Then, in order to exclude the finite intersection case, it is enough to prove that the supports of the divisors $T_\ell u(P)$ and $u^\sigma T_\ell(P)$ are not disjoint for infinitely many points $P$ on~$X_H$.
We prove that this happens for all the points $P$ in $S:=X_H(\CC)-X_H(\ol \QQ)$. Notice that $S$ is stable under automorphisms (which are defined over $\ol\QQ$) and Hecke operators, and all the points in $S$ are non-cuspidal and non-CM. Moreover, for each point $P\in S$, the support of $T_\ell(P)$ contains $\ell+1$ distinct points, and given two distinct points $P_1, P_2 \in S$, the supports of divisors $T_\ell(P_1)$ and $T_\ell(P_2)$  have at most one point in common.

We now fix $P\in S$ and consider $T_\ell (P) = R_1+\ldots + R_{\ell+1}$. 
Notice that there are $\ell+1$ distinct points (in $S$) whose image under $T_\ell$ contains $R_1$, and one of them is~$P$. Pick $Q$ to be another of these points and consider the divisor 
\[
D:= (T_\ell u - u^{\sigma} T_\ell)(P-Q)= T_\ell u(P)+ u^\sigma T_\ell (Q) - u^\sigma T_\ell (P) - T_\ell u(Q).
\]
Lemma \ref{th:commrel} tells us that $T_\ell u - u^{\sigma} T_\ell$ is zero on the Jacobian, hence that the above divisor is principal. Moreover, since the point $u^\sigma (R_1)$ cancels out in $u^\sigma T_\ell (Q) - u^\sigma T_\ell (P)$, then we can write $D = D^+- D^-$ with $D^+$ and $D^-$ effective divisors of degree $d\le 2\ell +1$. In particular, $d$ is strictly smaller than the gonality of $X_H$, which implies that the divisor $D$ is necessarily the zero divisor. Now consider the points $u^\sigma(R_2),\ldots, u^\sigma(R_{\ell+1})$, which appear in $u^\sigma T_\ell(P)$, and need to cancel out in $D$: they are not contained in the support of $u^\sigma T_\ell(Q)$ (indeed the supports of $T_\ell(P)$ and $T_\ell(Q)$ already share $R_{1}$, hence they cannot further intersect), and consequently they are contained in the support of $T_\ell(u(P))$. We have thus proved that $T_\ell(u(P))$ and  $u^\sigma T_\ell(P)$ intersect, as desired.
\end{proof}

We use the above result to prove our main result.

\begin{thm}\label{thm: prime level}
For each prime $p>311$ and each proper 
subgroup $H <  \GL_2(\ZZ/p\ZZ)$ satisfying $\det(H) = (\ZZ/p\ZZ)^\times$ and $-\Id \in H$, all the automorphisms of $X_H$ are modular.
\end{thm}

\begin{proof}
	If $H$ has order divisible by $p$, then, by \cite[Section~2.4, Proposition~15]{Ser72}, the group $H$ is contained in a Borel subgroup. In other words, up to conjugation, we can suppose that $H$ is contained in the subgroup $B < \GL_2(\ZZ/p\ZZ)$ of upper  
	triangular matrices that have order $p(p-1)^2$. Moreover, since $H$ has order divisible by~$p$, then it contains the $p$-Sylow of $B$, namely the subgroup $U:=\{ \smt1a01 : a\in \ZZ/p\ZZ\}$. 
	We deduce that $\Gamma_1(p)<\Gamma_H< \Gamma_0(p)$, and consequently $X_H(\CC) = \ol\HH/\Gamma_H$ is one of the curves studied by Kamienny in \cite{KAMIENNY_2003}. In this case, then, the result follows from Kamienny's characterization of the automorphisms of these curves, exactly when $p>311$.
	
	We can now assume that $H$ has order coprime to $p$. 
	We check that  the conditions in Theorem \ref{cor:CShighgen} are satisfied, implying our statement. 
	For the latter condition we bound the gonality of $X_H$ with \cite{Abr96} and \cite[discussion after Proposition~2, Appendix~2]{Kim03} which give 
	\begin{equation}\label{eq: gonality estimate Abramovich}    
		\mathrm{gonality}_{\ol{\QQ}}(X_H)=\mathrm{gonality}_\CC(X_H)\ge\frac{975}{4096}\cdot\frac{[\GL_2(\ZZ/p\ZZ) :\pm  H]}{24}.
	\end{equation}
	Since we are assuming that $H$, and consequently $\pm H$, has order coprime to~$p$, then, by \cite[Section~2.6]{Ser72}, $\pm H$ is contained either in the normalizer of a split Cartan subgroup, or in the normalizer of a non-split Cartan subgroup, or in an exceptional subgroup of $\GL_2(\ZZ/p\ZZ)$. In particular, 
	we have
	\begin{equation}\label{eq:technical estimate index}
		[\GL_2(\ZZ/p\ZZ) :\pm  H] \ge \min\left\{ \frac{p(p+1)}{2}, \frac{p(p-1)}{2}, \frac{p(p+1)(p-1)}{60}   \right\} \ge \frac{p(p-1)}{2},
	\end{equation}
	which implies, together with \eqref{eq: gonality estimate Abramovich}, that the gonality of $X_H$ is at least $8$. Hence, the second condition in Theorem \ref{cor:CShighgen} is satisfied with $\ell=3$.
	
	We now check the first condition in Theorem \ref{cor:CShighgen}. To bound the dimension of the CM part $J_H^\CM$ of $J_H$ we recall that $J_{H,\CC}$ is an isogeny factor of  $J_1(p^2)_\CC$: indeed  $X_H(\CC) = \ol\HH/\Gamma_H$ is a quotient of $X'(p)(\CC) = \ol\HH/\Gamma(p)$ with $X'(p)$ the modular curve associated to $\{ \smt *001\} < \GL_2(\ZZ/p\ZZ)$, which is covered by $X_1(p^2)$ through the map 
	$\psi\colon X_1(p^2) \to X'(p)$ sending a (non-cuspidal) point $(E,P)$ to the point $(E/\langle pP \rangle, E[p]/\langle pP \rangle, P)$. 
	In particular, we have 
	\begin{equation}\label{eq: from H to J_1p2}
		\dim\left(J_H^\CM\right) \le \dim\left(J_1(p^2)^\CM\right).    
	\end{equation}
	To bound $\dim(J_1(p^2)^\CM)$ we use \cite[Theorem~6.6.6]{DS05}, i.e., the isogeny decomposition $J_1(p^2) \sim \prod_{[f]} A_f^{m_f}$, where $f$, considered up to the action of Galois, varies among the newforms of level~$p$ and~$p^2$, with $m_f=2$ in the former case and $m_f=1$ in the latter. Each  $A_f$ is simple over~$\QQ$, hence it is either a CM abelian variety or it does not contain any CM factor over~$\ol \QQ$. By \cite[Proposition~1.6]{Shi72} and \cite{Shi71CM}, the first case happens exactly when $f$ is CM, that is associated to a primitive Gr\"ossencharacter of a quadratic imaginary field~$K$, say with conductor an ideal $\mathfrak m \subset\calO_K$. Then, by \cite[p.~168]{Shi72}, the level of $f$ is equal to $-\mathrm{disc}(K)\mathrm{Norm}(\mathfrak m)$.
	In our situation, $\mathrm{disc}(K)$ has to divide $p^2$, which implies that $K = \QQ(\sqrt{-p})$ with $p\equiv 3\bmod 4$, and that $\mathrm{Norm}(\mathfrak m)\in \{1,p\}$, that is  $\mathfrak m \in \{ (1), (\sqrt {-p})\}$. 
	We only have to consider $\frak m = (\sqrt {-p})$ since there is no  Gr\"ossencharacter with conductor  $\mathfrak m = (1)$: indeed we are  interested in Gr\"ossencharacters $\lambda$ of weight $\nu=1$ (so that $f$ has weight $2$) which, by \cite[p.~1]{Shi71CM}, satisfy 
	$$ 
	\lambda((\alpha))=\alpha, \quad \text{for each } \alpha \equiv 1 \pmod{\frak m},
	$$
	but this cannot be satisfied for both $\alpha = 1$ and $\alpha =-1$ (which are both congruent to $1$ modulo~$\mathfrak m=(1)$ and generate the same ideal $(\alpha)$).  
	
	In particular, all the CM factors $A_f$ in $J_1(p^2)$ are associated with forms of level $p^2$ and have multiplicity $m_f=1$.
	Moreover, $\dim A_f$ is equal to the number of Galois conjugates of $f$, which are again CM. We deduce that $\dim (J_1(p^2)^\CM)$ is equal to the number of CM newforms of level $p^2$, which is equal to the number of 
	Gr\"ossencharacters of $\QQ(\sqrt{-p})$ with conductor $\frak m= (\sqrt {-p})$. Since the ratio of two Gr\"ossencharacters is a character of the ray class group $\widetilde{\rm Cl}_{\frak m} (\QQ(\sqrt{-p}))$ of $\frak m$, then the number of our Gr\"ossencharacters is at most equal to the number of character of the abelian group $\widetilde{\rm Cl}_{\frak m} (\QQ(\sqrt{-p}))$, which is in turn equal to the cardinality of the group $\widetilde{\rm Cl}_{\frak m}( \QQ(\sqrt{-p}))$ itself. Then, using the estimates in \cite[Theorem~8.10 and Lemma~8.16]{Nar04} to bound the (usual) class group of  $\QQ(\sqrt{-p})$, we obtain 
	\begin{equation}\label{eq: estimate CM J_1p2}
		\dim J_1(p^2)^\CM \le \# \widetilde{\rm Cl}_{\frak m} (\QQ(\sqrt{-p})) = \# \frac{(\calO_{\QQ(\sqrt{-p})}/\mathfrak m)^\times}{\calO_{\QQ(\sqrt{-p})}^\times} \cdot   \# \mathrm{Cl}(\QQ(\sqrt{-p}))
		\le 
		\frac{p-1}{2} \cdot \frac 1\pi (2+\log p) \sqrt{p}.
	\end{equation}
	For $p>100$, Equations~\eqref{eq: from H to J_1p2} and \eqref{eq: estimate CM J_1p2} give 
	\[
	\dim J_H^\CM \le \frac{p-1}{2} \frac{2+\log p}{\pi} \sqrt{p}  \le \frac{1}{4} p^{3/2}\log p -1.
	\]
	The other quantity appearing in the first condition of Theorem \ref{cor:CShighgen} is the dimension of $J_H$, that is the genus of $X_H$. For $p>311$, the genus is at least $\frac{p(p-10)}{24}$: when $H$ is contained in an exceptional subgroup this follows from the fact that the genus is at least equal to the gonality, which we bound using \eqref{eq: gonality estimate Abramovich} and the inequality $[\GL_2(\ZZ/p\ZZ):H]\ge p(p-1)(p+1)/60$; when $H \in \{ C_\ns^+(p), C_\s^+(p)\}$ is the normalizer of a Cartan subgroup, the desired bound follows from the usual formula for the genus (see e.g., \cite[Equation~3.13 and Table~1]{DLM22}); for $H$ contained in $C_\ns^+(p)$, respectively $C_\s^+(p)$, we notice that $X_H$ has a map to $X_{C_\ns^+(p)}$, respectively $X_{C_\s^+(p)}$, and consequently its genus is at least that of $X_{C_\s^+(p)}$.
	We deduce that, for $p>2^{14}$, we have the desired inequality, namely  
	\[
	\mathrm{genus}(X_H)\ge 
	\frac{p(p-10)}{24} \ge \frac 1 2 p^{3/2} \log p \ge 2 \dim J_H^\CM +2.
	\]
	We check the first condition of Theorem \ref{cor:CShighgen} in the remaining cases as follows. 
	We only need to check the primes $311<p<2^{14}$ that are congruent to $3$ modulo~$4$. Moreover, by \cite[Theorem~5.11]{DLM22} and \cite{Gon16}, we can exclude the cases $H\in\{C_\ns^+(p),C_\s^+(p)\}$. Then, either $H$ is contained in an exceptional subgroup or $X_H$ has a map of degree $d\ge 2$ to  $X_{C_\ns^+(p)}$ or $X_{C_\s^+(p)}$. In the first case, the genus is at least $\frac{p(p-1)(p+1) \cdot 975}{60^\cdot 24 \cdot 4096 }$ using again the bound \eqref{eq: gonality estimate Abramovich} on the gonality; in the second case, using the Riemann-Hurwitz formula, we get that the genus is at least $\frac{p(p-10)}{12}-1$. We then check prime-by-prime that the genus is greater than $2\dim(J_H^\CM)+1$ using the inequality $\dim(J_H^\CM) \le  \frac{p-1}{2} \# \mathrm{Cl}(\QQ(\sqrt{-p}))$, as before, and computing the class number. 
\end{proof}

\begin{rem}
For general level $n$ we have the following partial results.
\begin{itemize}
\item    Condition \ref{condition_CM} in Theorem \ref{th:highgenus} is satisfied whenever $X_H$ has genus at least $2$ and every prime $p\mid n$ satisfies $p\equiv 1 \pmod 4$. Indeed, under these assumptions we have $J_H^\mathrm{C}=0$. Suppose $J_H$ has non-trivial CM-part. Using the decomposition of $J_H$ into factors $A_f$, the CM part comes from factors $A_f$ with $f$ CM newforms of level dividing $n^2$ (\cite[Proposition~1.6]{Shi72} and \cite{Shi71CM}). Picking one of those $f$ with CM by a certain imaginary field $K$, by \cite[p.~168]{Shi72} the discriminant of $K$ divides the level of $f$. This implies that such a discriminant is of the form $- \prod_{p\mid n} p^{e_p}$, which is congruent to $-1$ modulo $4$, which is absurd.
\item Say that two curves $X_H$ and $X_K$ are equivalent if $\Gamma_H = \Gamma_K$ (this equivalence implies that  $X_H$ and $X_K$ are isomorphic over $\ol\QQ$ and happens, e.g., in the context of \cite[Section~3.2]{Twists_geom_derickxhashimotoNajman}). Then, for all but finitely many equivalence classes of geometrically connected modular curves, 
Condition \ref{condition_gonality} is satisfied (up to choosing wisely the curve in the equivalence class). 

To see it, define \emph{geometric level} of an equivalence class $[X_H]$ of modular curves as the smallest $m$ such that $\Gamma_H$ contains $\Gamma(m)$. Then, we have $[X_H] = [X_{H'}]$ for $X_{H'}$ a geometrically connected modular curve of level a (divisor of) a power of $m$: write $n = n'r$, with $n'$ the maximal divisor of $n$ containing the primes appearing in $m$, so that $r=n/n'$  is prime to $n'$ and, under the canonical isomorphisms $\mathrm{SL}_2(\ZZ/n \ZZ) = \mathrm{SL}_2(\ZZ/n'\ZZ) \times \mathrm{SL}_2(\ZZ/r\ZZ)$ and $\mathrm{GL}_2(\ZZ/n \ZZ) = \mathrm{GL}_2(\ZZ/n'\ZZ) \times \mathrm{GL}_2(\ZZ/r\ZZ)$, then $H \cap \mathrm{SL}_2(\ZZ/n \ZZ) $ is of the form $K \times \mathrm{SL}_2(\ZZ/r \ZZ)$, and if we take $H'$ to be the projection, modulo $n'$, of $ H \cap ( \GL_2(\ZZ/n'\ZZ) \times \{\Id \} )$, then $H'$ has surjective determinant and satisfies $H'\cap \mathrm{SL}_2(\ZZ/n'\ZZ)  = K$. 
Then,  \cite[Equation (4.22)]{CP84} and \cite{Abr96} give a lower bound on the genus,  hence on the gonality, of $ X_{H', \ol \QQ}$ which is linear in $m$ and, for $m  \ge 1926$, one can prove that there is a prime $\ell$ coprime to $m$, hence to $n'$, satisfying the required inequality.  
\end{itemize}   
\end{rem}

\section{Explicit determination of the automorphism group}
\label{sec: algos}

\subsection{Computing modular automorphisms}\label{sec:MAut}

In this section, we compute the modular automorphisms of the modular curves $X_H$, when $\det H= (\ZZ/n\ZZ)^\times$. More generally, for every $H$, we compute the modular automorphisms of $\Gamma_H\backslash \ol\HH$ (see Definition \ref{def:modaut}) which form a subgroup of $\Aut(\Gamma_H \backslash \overline\HH )$, namely 
\begin{equation} \label{eqMAut}
\mathrm{MAut}(\Gamma_H \backslash \overline\HH) \cong  \mathcal N/\PP \Gamma_H,
\end{equation}
where $\mathcal N$ is the normalizer of $\PP \Gamma_H$ inside $\mathrm{PGL}_2^+(\QQ)$.

We want, then, to compute generators for the above quotient. We work in $\PGL_2^+(\QQ) = \GL_2^+(\QQ)/\QQ^\times$, and, for simplicity, we choose representatives in $\GL_2^+(\QQ)$ being reduced in the following sense: $m \in\GL_2(\QQ)$ is \emph{reduced} if 
\begin{equation}\label{eq:reduced}
m= \smt abcd \in \ZZ^{2\times 2} \cap \GL_2(\QQ) \quad \text{and} \quad \gcd(a,b,c,d)=1 . 
\end{equation}

The following proposition states that, in our computations, we can look at the normalizer of $\Gamma_H$ instead of $\PP\Gamma_H$ and we can reduce to a finite number of checks by looking at matrices modulo~$n^2$.

\begin{prop}\label{prop:finitecases}
Let $H$ be a subgroup of $\GLn$ containing $-\Id$ and let $\calN = \Norm_{\PGL_2^+ (\QQ)}\PP \Gamma_H$.
\begin{enumerate}
\item \label{propitem:P=GL}  A matrix $m\in \GL_2 (\QQ)$ normalizes $\Gamma_H$  if and only if its class $m \in \PGL_2(\QQ)$  normalizes $\PP\Gamma_H$; hence $\calN = \PP\big(\Norm_{\GL_2^+(\QQ)}\Gamma_H\big)$. 
\item \label{item:det} If $m$ is a reduced matrix (see Equation \eqref{eq:reduced}) normalizing $\Gamma_H$, then 
\[
\det(m)\mid n.
\]
\item \label{item:test_m} Let $\gamma_1, \ldots, \gamma_k \in \Gamma_H$ be a full set of representatives for $\Gamma_H/\Gamma(n^2)$.
Then a reduced matrix $m$ with determinant dividing $n$ normalizes $\Gamma_H$ if and only if 
\begin{equation} \label{eq:test_m}
m \gamma_i m^{-1}, m^{-1} \gamma_i m \in \Gamma_H \text{ for all } i=1,\ldots, k.
\end{equation}
\item  \label{item:nD_then_equiv} Let $m,m'$ be two reduced matrices of determinant~$D$. If $m \equiv m'$ modulo~${nD}$, then they lie in the same coset $m\Gamma_H = m'\Gamma_H$. 
In particular $m$ normalizes $\Gamma_H$ if and only if $m'$ normalizes $\Gamma_H$ and, if this happens, $m$ and $m'$ define the same modular automorphism.
\end{enumerate}
\end{prop}
\begin{proof}
For Part~\ref{propitem:P=GL}, we have to see that if $m$ normalizes $\PP \Gamma_H$, then it normalizes $\Gamma_H$ (the other implication is easy). Given $\gamma \in \Gamma_H$, we need to show that $m\gamma m^{-1} \in \Gamma_H$: by hypothesis, after projectivization it lies in $\PP\Gamma_H$, i.e., it is of the form $m\gamma m^{-1} = q\gamma'$ for $q$ a rational number and $\gamma'\in \Gamma_H$; since $1 = \det(\gamma) = \det(m\gamma m^{-1}) = q^2 \det(\gamma') = q^2$, then $q=\pm 1$, hence what we wanted.

For Part~\eqref{item:det} it is sufficient to notice that $\smt 1n01$ and $\smt 10n1$ are elements of $\Gamma_H$, hence $m \smt 1n01 m^{-1}$ and $m \smt 10n1 m^{-1}$ lie in $\Gamma_H \subset \SLZ$. If $D := \det(m)$ and $m = \smt abcd$, writing explicitly we get
\[
m \smt 1n01 m^{-1}  = \tfrac1D \smt{D-nac}{na^2}{-nc^2}{D+nac},\quad m \smt 10n1 m^{-1} =  \tfrac1D\smt{D+nbd}{-nb^2}{nd^2}{D-nbd} \quad  \in \quad\SLZ,
\]
implying that $D\mid na^2,\, D\mid nb^2, \, D\mid nc^2, \, D \mid  nd^2$ and hence that $D \mid n{\cdot}\gcd(a^2,b^2,c^2,d^2) = n$.

Now Part~\eqref{item:test_m}. If $m$ normalizes $\Gamma_H$, Equation \eqref{eq:test_m} is true by definition.
For the other implication, suppose that $m$ satisfies \eqref{eq:test_m} and let $D := \det(m)$. Given $\gamma$ in $\Gamma_H$, we want to prove that $m\gamma m^{-1}$ and $m^{-1}\gamma m$ lie in $\Gamma_H$. We prove it for $m\gamma m^{-1}$. Indeed, by definition of the $\gamma_j$'s, we can choose $\gamma_i$ congruent to $\gamma$ modulo~$n^2$ and hence modulo~$nD$.
By hypothesis, $m\gamma_i m^{-1}$ lies in $\Gamma_H$, therefore if we write $m^{-1} = \tfrac 1D \overline m$ where $\overline m$ is the adjoint matrix, we have
\[
m\gamma_i \overline m = D \gamma_i'\quad \text{for certain }\gamma_i'\in \Gamma_H.
\]
Since $\gamma \equiv \gamma_i \mod{nD}$ and $\overline m$ has integral entries, we have
\[
m\gamma \overline m \equiv  D \gamma_i'  \mod{nD}.
\]
Hence
$m\gamma m^{-1} = \tfrac1D m\gamma \overline m$
has integral entries and is congruent to $\gamma_i'$ modulo $n$, therefore it belongs to~$\Gamma_H$.

Part~\eqref{item:nD_then_equiv} follows again by the fact that $\overline m 
= D m^{-1}$ has integral entries and by the congruence:
\[
D m^{-1} m'\equiv \smt D00D \mod{nD}.
\]
Hence $m^{-1}m'$ has integral entries and lies in $\Gamma(n) \subset \Gamma_H$. 
\end{proof}

The above proposition tells us that to compute the normalizer of $\Gamma_H$ we only need to look at finitely many matrices and check if they normalize. In certain cases we are able to reduce significantly the number of checks to perform. We now discuss these cases.

\begin{prop}\label{prop:modborel}
Let $H$ be a subgroup of $\GLn$ containing $-\Id$  and let $m$ be a reduced matrix (see Equation \eqref{eq:reduced}) normalizing $\Gamma_H$.
Let $n=  \prod _{p\mid n} p^{n_p}$ and $\det m = \prod _{p\mid n} p^{e_p}$ be the respective prime factorizations and let $H' = H \cap \SLn$. 
Then:
\begin{enumerate}
\item \label{item:no_Bor} If the image of $H'$ modulo $p^k$ is not contained in a Borel subgroup of $\GL_2(\ZZ/p^k\ZZ)$, then $e_p<k$.
\item \label{item:Borp} If the image of $H'$ modulo $p$ is contained in the standard Borel subgroup $\{\smt**0*\}$, but not in a split Cartan subgroup, then either $e_p=0$ or $m\equiv \smt0*00$ modulo $p$.
\item \label{item:Borpk} If the image of $H'$ modulo $p^k$ is contained in the standard Borel subgroup $\{\smt**0*\}$ and, moreover, the image $H_p$ of $H'$ modulo~$p$ is not contained in a split Cartan subgroup nor in $\{\smt a*0a: a\in (\ZZ/p\ZZ)^\times\}$, then either  $e_p=0$ or $m\equiv \smt0*00$ modulo $p^k$.
\item \label{item:Car} If the image of $H'$ modulo $p^k$ is contained in a split Cartan subgroup, then $e_p \le n_p-k$.
\end{enumerate}
\end{prop}
\begin{proof}
Write $m = \smt abcd$. For Parts \eqref{item:no_Bor} and \eqref{item:Borp} let $\Lambda \subset \ZZ^2$ be the lattice generated by $\lambda_1=\left( \begin{smallmatrix} a \\ c \end{smallmatrix}\right)$ and $ \lambda_2=\left( \begin{smallmatrix} b \\ d \end{smallmatrix}\right)$. 
Since $m$ normalizes $\Gamma_H$, for every $\gamma \in \Gamma_H$ there is $\gamma' = \smt{x}{y}{z}{w}\in\Gamma_H$ such that $\gamma m = m \gamma'$. Hence, looking at the columns of $\gamma m$, we get $\gamma \lambda_1 = x\lambda_1{+}z\lambda_2 $ and $\gamma \lambda_2 =  y\lambda_1{+}w\lambda_2$.
Since $\gamma$ is arbitrary in $\Gamma_H$ and $\gamma'$ has integral entries, we have $\Gamma_H \Lambda \subset \Lambda$. Actually each matrix in $\Gamma_H$ is invertible, hence
\begin{equation} \label{eq:Lambda_m_invariant}
\Gamma_H \Lambda = \Lambda.
\end{equation}
Consider also $\overline \Lambda = \Lambda \bmod p$, i.e., the image of $\Lambda$ under the quotient map $\ZZ^2 \to (\ZZ/p\ZZ)^2$. Since $m$ is reduced, then $\overline\Lambda \ne \{ 0\}$. Moreover, if $p$ divides $\det(m)$, then $\overline \Lambda \ne (\ZZ/p\ZZ)^2$, hence $\overline \Lambda$ is a line inside $(\ZZ/p\ZZ)^2$. Because of Equation \eqref{eq:Lambda_m_invariant}, this line is invariant by (the reduction $\bmod\; p$ of) every matrix in $\Gamma_H = \Gamma_{H'}$, hence by the image of $H'$ modulo~$p$. We proved that if $p$ divides $\det(m)$, the image of $H'$ modulo~$p$ is contained in a Borel subgroup of $\GL_2(\ZZ/p\ZZ)$. Analogously if $p^k$ divides $\det(m)$, then the image of $\Lambda$ modulo~$p^k$ is a line in $(\ZZ/p^k\ZZ)^2$ that is $\Gamma_{H'}$-invariant. Since the reduction map $\Gamma_{H'} \to H'$ is surjective, the image of $H'$ modulo~$p^k$ has an invariant line, which is equivalent to saying that the image of $H'$ in $\GL_2(\ZZ/p^k\ZZ)$ is contained in a Borel, proving Part~\eqref{item:no_Bor}.

With the hypothesis of Part~\eqref{item:Borp} we notice that the image of $H'$ modulo~$p$ has only one invariant line in $(\ZZ/p\ZZ)^2$, namely $\mathrm{Span}\left\{\left( \begin{smallmatrix} 1 \\ 0 \end{smallmatrix}\right)\right\}$. This implies that $\overline \Lambda =\mathrm{Span}\left\{\left( \begin{smallmatrix}	1 \\ 0 \end{smallmatrix}\right)\right\}$, i.e., that $c\equiv d \equiv 0 \pmod p$. Moreover $\det(m)m^{-1} = \smt d{-b}{-c}a$ is a reduced matrix that normalizes $\Gamma_H$ and, if we apply the same argument, we see that $a\equiv 0 \pmod p$ as well, proving Part~\eqref{item:Borp}.

For Part~\eqref{item:Borpk} we refine the argument before Equation \eqref{eq:Lambda_m_invariant}. Let $\lambda_1, \lambda_2$ be the columns of $m$. Because of Part~\eqref{item:Borp} we know that  $\lambda_1 \equiv \vv00 \pmod p$ and $\lambda_2 \pmod p$ is a non-zero multiple of $\vv10$ (non-zero otherwise $m$ would not be reduced). Let $h$ be the ``$p$-adic valuation of $\lambda_1$,'' that is the maximal exponent such that $\lambda_1 = p^h \lambda_1'$ for $\lambda_1'\in\ZZ^2$ satisfying $ \lambda_1' \not\equiv \vv00 \pmod p$. 
It is enough to show that $h \ge k$, in fact this implies that $m =\smt abcd \equiv \smt0*0* \pmod{p^k}$ and if we then apply the same argument to the reduced matrix $\det(m)m^{-1} = \smt{d}{-b}{-c}a$, we see that $d \equiv 0 \pmod{p^k}$. 

Suppose by contradiction that $h<k$. For each $\gamma \in \Gamma_H$ we consider the conjugate matrix
\[
m^{-1}\gamma m = \tilde \gamma =  \smt{\tilde \gamma_{11}}{\tilde \gamma_{12}}{p^{k} \tilde \gamma_{21}}{\tilde \gamma_{22}}.
\]
    For each $\gamma\in \Gamma_H$, since $\gamma m = m \tilde \gamma$, by looking at the columns we see 
\begin{equation*}
p^h \gamma \lambda_1' =  p^h  \tilde \gamma_{11}\lambda_1' + p^k \tilde \gamma_{21}\lambda_2, \qquad 
\gamma \lambda_2 = p^h \tilde \gamma_{12}  \lambda_1' +  \tilde \gamma_{22} \lambda_2.
\end{equation*}
Since $k>h\ge 1$, we deduce that 
\begin{equation} \label{eq:congr_colslambda}
\gamma \lambda_1' \equiv  \tilde \gamma_{11}  \lambda_1' \pmod p, \qquad 
\gamma \lambda_2 \equiv \tilde \gamma_{22} \lambda_2  \pmod p.
\end{equation}
We now prove that $H_p$ contains the subgroup $U= \{\smt 1*01\}$. If by contradiction $U \cap H_p$ was not equal to $U \cong \ZZ/p\ZZ$, then $U \cap H_p$ would be trivial,  implying the injectivity of the following map 
\begin{align*}
H_p &\to (\ZZ/p\ZZ)^\times \\
\smt ab0d &\mapsto a .
\end{align*}
In particular $H_p$ would embed in a cyclic group of order $p-1$, and hence it would be itself cyclic, generated by a matrix $g$ of order dividing $p-1$. In particular the minimal polynomial of~$g$ would divide $x^{p-1}-1$, hence it would be separable, implying that $g$ would be a diagonalizable matrix and $H_p$ would be contained in a split Cartan subgroup. This contradiction tells us that $H_p$ contains~$U$.

In particular, we choose a matrix $\gamma \equiv \smt1101 \pmod p$ in $\Gamma_H$, so that the only eigenvectors of \mbox{$\gamma \pmod p$} are the non-zero multiples of $\vv10$. Equation \eqref{eq:congr_colslambda} then implies that $\lambda_1',\lambda_2 \pmod p$ are both non-zero multiples of $\vv10$, hence that $\lambda_1' \equiv  \delta  \lambda_2 \pmod p$ for some $\delta \in (\ZZ/p\ZZ)^\times$. Since $H_p$ is not contained in $\{\smt a*0a: a\in (\ZZ/p\ZZ)^\times\}$ and since the conjugation by $m$ is surjective on $\Gamma_H$, we can find a $\gamma$ such that \mbox{$\tilde \gamma \pmod p$} does not belong to $\{\smt a*0a: a\in (\ZZ/p\ZZ)^\times\}$. In other words, we can take $\tilde\gamma_{11} \not\equiv \tilde\gamma_{22} \pmod p$. This is a contradiction, since, by Equation \eqref{eq:congr_colslambda}, we have 
\[
\delta \tilde \gamma_{11} \lambda_2 \equiv  \tilde \gamma_{11} \lambda_1' \equiv \gamma\lambda_1' \equiv \delta \gamma \lambda_2\equiv \delta \tilde \gamma_{22} \lambda_2 \pmod p.
\]

For Part~\eqref{item:Car}, up to conjugation in $\SL_2(\ZZ)$ we can assume that the image of $H'$ modulo~$p^k$ is contained in the standard split Cartan subgroup $\{\smt*00*\}$. Writing $D = \det(m)$ and $m = \smt abcd$, as in Proposition \ref{prop:finitecases} Part~\eqref{item:det}, we look at the matrices $\smt 1n01$ and $\smt 10n1$ and we notice that the conjugate matrices
\[
m \smt 1n01 m^{-1} = \tfrac1D \smt{D-nac}{na^2}{-nc^2}{D+nac}\quad\text{and }\quad m \smt 10n1 m^{-1} = \tfrac1D\smt{D+nbd}{-nb^2}{nd^2}{D-nbd} 
\]
must have integral entries and be congruent to $\smt*00*$ modulo $p^k$. Since at least one among $a,b,c,d$ is not a multiple of~$p$, we immediately deduce the conclusion.
\end{proof}

\begin{rem}
In Proposition \ref{prop:modborel}, the hypotheses in the various parts can be checked ``geometrically'':
\begin{itemize}
\item The image of $H'$ modulo~$p^k$ is contained in a Borel subgroup if and only if there exists a vector $v\in (\ZZ/p^k\ZZ)^2$ that is stable under $H'$ (equivalently it is a common eigenvector of generators of $H'$) and whose reduction modulo~$p$ is non-zero.

Notice that if $H'$ modulo~$p^k$ is contained in a Borel subgroup, it is easy to find a base change after which $H'$ is contained in the standard Borel: if $v$ is a vector as before, we can lift it to a primitive vector in $\ZZ^2$ (primitive means having coprime entries), and every basis $\mathcal B = \{v,w\}$ of $\ZZ^2$ will work.

\item The image of $H'$ modulo~$p^k$ is contained in a split Cartan subgroup if and only if there exist two vectors $v_1,v_2\in (\ZZ/p^k\ZZ)^2$ that are stable under $H'$ and whose reductions modulo~$p$ are linearly independent.
\item In Part~\eqref{item:Borpk}, after checking that $H'$ is contained in a Borel subgroup, checking that $H'$ is not contained in (a conjugate of) $\{\smt a*0a: a\in (\ZZ/p\ZZ)^\times\}$ modulo~$p$ is equivalent to checking that at least one generator of $H_p$ has distinct eigenvalues. 
Notice that Part~\eqref{item:Borpk} in the case $k=1$ is weaker than Part~\eqref{item:Borp}.
\end{itemize}
\end{rem}

Now we describe an algorithm to compute a set of representatives with coprime integer entries of $\mathcal N/\PP \Gamma_H$ (see Equation \eqref{eqMAut}). As explained above, this set is actually a group, namely the modular automorphism group. An implementation of the following algorithm can be found in \cite[file \texttt{\detokenize{find_modular_automorphisms.txt}}]{script_MAut}.

\begin{Algorithm}\label{alg:modaut} Given a subgroup $H<\GLn$ containing $-\Id$, compute the set $\mathrm{MAut}(\Gamma_H\backslash \overline{\HH})$, i.e., a set of representatives of all modular automorphisms of $\Gamma_H\backslash \overline{\HH}$.
\begin{description}
\item[Step 1 (Possible determinants)] Find the set $\mathcal P$ of all the primes $p\mid n$ such that $H\cap\mathrm{SL}_2(\ZZ/n\ZZ)$ is contained in a Borel subgroup modulo $p$. Moreover, find the partition 
\[
\mathcal P  = \mathcal P_B \sqcup \mathcal P_C \sqcup \mathcal P_O,
\]
where:
\begin{itemize}
\item $\mathcal P_B$ is the set of primes such that the image of $H\cap \SLn$ modulo~$p$ is not contained in a split Cartan subgroup nor in a conjugate of $\{\smt a*0a\}$ (see Condition~(\ref{item:Borpk}) in Proposition \ref{prop:modborel}).
\item $\mathcal P_C$ is the set of primes such that the image of $H\cap \SLn$ modulo~$p$ is contained in a split Cartan subgroup (see Condition~(\ref{item:Car}) in Proposition \ref{prop:modborel}); 
for $p\in \mathcal P_C$ denote $c_p$ the maximum integer such that the image of $H \cap \SLn$ modulo~$p^{c_p}$ is contained in a split Cartan subgroup.
\item $\mathcal P_O$ contains the remaining elements of $\mathcal P$.
\end{itemize}
Let $p^{n_p}$ be the maximal power of $p$ dividing $n$ and for each $p\in \mathcal P$ denote $k_p$ the maximum integer such that the image of $H \cap \SLn$ modulo~$p^{k_p}$ is contained in a Borel subgroup. 

Compute the set $\mathcal D$ containing all integers of the form $\prod_{p\in \mathcal P} p^{e_p}$ with $e_p \in \{0,k_p\}$  if $p\in \mathcal P_B$, and  $e_p \in \{0, \ldots, \min(k_p, n_p-c_p)\}$ if $p\in \mathcal P_C$ and $e_p \in \{0, \ldots, k_p\}$ if $p\in \mathcal P_O$.

\item[Step 2 (Conjugation)] Compute $\gamma\in\SLZ$ such that
\[
H':=\gamma^{-1} \left( H\cap\mathrm{SL}_2(\ZZ/n\ZZ) \right)\gamma
\]
is contained in the subgroup of the upper triangular matrices modulo~$p^{k_p}$ for all $p \in \mathcal P_B \cup  \mathcal P_O$.

\item[Step 3 (Normalizer in $\SL_2$)] Compute the normalizer $N'$ of $H'$ in $\mathrm{SL}_2(\ZZ/n\ZZ)$.

\item[Step 4 (Quotient by $N'$)]  Consider the action of (the inverse image in $\SL_{2}(\ZZ/n^2\ZZ)$ of) $N'$ on row vectors $(c,d)\in (\ZZ/n^2\ZZ)^2$: compute a set $\mathcal R$ of representatives for the orbits of this action. 

\item[Step 5 (Check all candidates)] Let $\Xi:=\{ \Id \}$ and for each $D$ in $\mathcal D {-}\{1\}$ do the following:
\begin{description}
\item[5.a] Set $\Delta = \gcd\left( D, (\prod_{p\in \mathcal P_{B}} p^{k_p})(\prod_{p\in \mathcal P_{O} } p) \right)$.
\item[5.b] Take $\mathcal R_D$ to be the set of reductions modulo $nD$ of elements in $\mathcal R$ of the form $({\Delta c}, {\Delta d})$. For each $({\Delta c}, {\Delta d}) \in \mathcal R_D $ find all possible $a,b \in \ZZ/nD\ZZ$ such that the matrix
\[
m = \smt{\Delta a}{b}{\Delta c}{\Delta d}
\]
has determinant $D$ and $\gcd(n,\Delta a,b,\Delta c,\Delta d)=1$.
\item[5.c]\label{step_lift_check} For all matrices $m$ found above, compute  a reduced lift $\tilde m\in \ZZ^{2\times 2}$ having determinant $D$. If $ \tilde m \Gamma_{H'}  \tilde m^{-1}= \Gamma_{H'}$ (check it using the criterion in Equation \eqref{eq:test_m}), add $\tilde m$ to $\Xi$.
\end{description}
\item[Step 6 (Discard possible repetitions)] If $m_1,m_2\in \Xi$ such that $m_1\ne m_2$ and $m_1 \Gamma_{N'} = m_2\Gamma_{N'}$, discard $m_2$. 
\item[Output] $\Xi$ together with the (lifts of the) elements of $N'/H'$ generates the modular automorphism group: for each $\nu \in N'/H'$ choose a lift $\tilde \nu \in \SLZ$ and output
\[
\{ \gamma (\xi \tilde\nu) \gamma^{-1} \,:\, \xi \in \Xi , \nu \in N'/H' \}.
\]
\end{description}
\end{Algorithm}

\begin{proof}[Correctness of Algorithm \ref{alg:modaut}]
We aim to show that every modular automorphism can be represented by an element in $\{ \gamma \xi \tilde \nu \gamma^{-1} \,:\, \xi \in \Xi , \nu \in N'/H' \}$.
Up to a base change, it is enough to prove that the set $\{ \xi \tilde \nu  \,:\, \xi \in \Xi , \nu \in N'/H' \}$ surjects onto all modular automorphisms of $\Gamma_{H'} \backslash \overline{\HH}$, where $H' = \gamma^{-1}(H\cap \SL_2(\ZZ/n\ZZ))\gamma$ as in Step~$2$: indeed if we define $H_1 = H\cap \SL_2(\ZZ/n\ZZ)$, then $\Gamma_H = \Gamma_{H_1} = \gamma \Gamma_{H'}\gamma^{-1}$ and the same holds for the normalizer.

Because of Proposition \ref{prop:finitecases} Part~\eqref{propitem:P=GL}, to compute the modular automorphisms of $\Gamma_{H'} \backslash \overline{\HH}$ is equivalent to compute the matrices normalizing $\Gamma_{H'}$ up to scalars. Since we work up to scalars, we only need to look at reduced matrices (see Equation \eqref{eq:reduced}).
We start considering the reduced matrices $m$ with  $\det m =1$. This case is dealt with in Step~$3$: since the reduction map $\SLZ \to \SL_2(\ZZ/n\ZZ)$ is surjective, then a matrix in $\SLZ$ normalizes $\Gamma_{H'}$ if and only if its reduction normalizes $H'$,
hence the third step computes all the reductions of matrices $m = \tilde\nu$ of determinant $1$ normalizing $\Gamma_{H'}$.

Notice that $\Gamma_{N'}$ is a subgroup of the normalizer $\calN$ of $\Gamma_{H'}$ (see Equation \eqref{eqMAut}). Therefore it gives the following subgroup of the modular automorphisms:  $\Gamma_{N'}/\Gamma_{H'}$, which is made of all the classes of matrices $\tilde\nu$ considered in the Output. In particular, to prove that the algorithm is correct, it is enough proving that $\Xi$ contains a full set of representatives for $\PP \calN / \PP \Gamma_{N'}$. Each class in this quotient contains reduced matrices and, by the argument above, the only reduced representative of determinant $1$ is the identity matrix which lies in $\Xi$. We proceed with the case $\det m >1$.

Because of Proposition \ref{prop:finitecases} Part~\eqref{item:det}, it is enough checking reduced matrices of determinant $D \mid n$. Actually Proposition \ref{prop:modborel} tells us that not all determinants $D$ need to be checked. For example if the image modulo~$p$ of $H'$ is not contained in a Borel subgroup, then we don't need to check those $D$'s that are divisible by~$p$. Moreover, by zooming in on the images modulo~$p^k$ of $H\cap \SL_2(\ZZ/n\ZZ)$ and using Proposition \ref{prop:modborel}, we further reduce to check the $D$'s belonging to the set $\mathcal D$ in Step~$1$: in fact, by Part~\eqref{item:no_Bor}, the $p$-adic valuation of $D$ must be at most $k_p$; moreover, by Part~\eqref{item:Car} of the proposition, if $H' \bmod p$ is contained in a split Cartan subgroup, the valuation is also at most $n_p-c_p$; finally if $p\in \mathcal P_B$, Part~\eqref{item:Borpk} tells us that the valuation must be either $0$ or~$k_p$.

In Step~$5$ we go through all possible $D \in \mathcal D$, except $D=1$. We consider ``all'' the possible reduced matrices $m$ of determinant $D$ and check whether they normalize $\Gamma_{H'}$ or not using Proposition \ref{prop:finitecases}, Part~\eqref{item:test_m} (see Step~$5.c$).
By Proposition \ref{prop:finitecases}, Part~\eqref{item:nD_then_equiv}, it is enough to work ``modulo $nD$,'' i.e., it is enough to check a finite set of matrices $\{ m_1, \ldots, m_r \}$ whose reduction modulo~$nD$ surjects on the set $\{\overline m \in (\ZZ/nD\ZZ)^{2\times 2}: \det\overline m = D \}$: indeed two matrices with the same reduction modulo~$nD$ either they define the same modular automorphism, or neither of them normalize. 
More than this, since we are looking modulo~$\Gamma_{N'}$, it is enough to check a full set of representatives for
\begin{equation}\label{eq:a_quotient_of_matrices}
\{\overline m \in (\ZZ/nD\ZZ)^{2\times 2}: \det\overline m = D \}/ \Gamma_{N'} = \{\overline m \in (\ZZ/nD\ZZ)^{2\times 2}: \det\overline m = D \}/ \sim_{N'},
\end{equation}
where $m \sim_{N'} m'$ if there is a matrix $\nu$ in (the lift in $\SL_2(\ZZ/nD\ZZ)$ of) $N'$ such that $m' = m \nu$. Since multiplication by $\nu$ is equivalent to multiplying the rows of $m$ by $\nu$, it is enough to check matrices $m$ whose bottom row is in (the reduction modulo $nD$ of) $\mathcal R$ (see Step~$4$). In fact $\mathcal R$ is defined as a full set of representatives for $(\ZZ/n^2\ZZ)^2/N'$, hence their reductions give a full set of representatives for $(\ZZ/nD\ZZ)^2/N'$. 

In the algorithm we do not check all the classes in Equation \eqref{eq:a_quotient_of_matrices} but we reduce to a smaller set of matrices. 
We have chosen $\gamma$ and $H'$ so that when the image modulo~$p^k$ of $H'$ is contained in a Borel subgroup but not in a split Cartan subgroup,  then it is contained in the standard Borel subgroup $\{\smt **0*\}$ modulo~$p^k$. Proposition \ref{prop:modborel} then implies that it is enough to check matrices such that all entries except the top-right are multiple of $\Delta$. Therefore, among all matrices whose bottom row are (reductions of) rows $r\in \mathcal R$ we can discard the ones with $r$ not multiple of $\Delta$. In particular, as in Step~$5.b$, it is enough looking at matrices $m$ with bottom row in $\mathcal R_D$, with determinant~$D$ and with the entry $m_{1,1}$ divisible by~$\Delta$.
All these matrices are checked in Step~$5.c$, hence the algorithm is correct.
\end{proof}

\begin{rem}
In Step~$5.c$ we lift matrices modulo~$nD$ to reduced integral matrices. At the moment the implementation chooses a random lift, not one with ``small'' entries. For example in the case where $H$ is a Borel subgroup, it happens that the output of the algorithm does not contain the standard Fricke involution $\smt 0{-1}n0$, but another representative in the coset $\smt 0{-1}n0\Gamma_{N'} $.
\end{rem}

The above algorithm outputs the modular automorphisms as a set.
To compute the group structure of the automorphisms, one can compute the permutation representation as follows.

\begin{Algorithm}\label{alg:modaut_2}
Given $H<\GLn$ containing $-\Id$ and the set $\mathrm{MAut}$ of modular automorphisms computed in Algorithm \ref{alg:modaut}, compute the group structure.
\begin{description}
\item[Step 1]
Let $r:=\#\mathrm{MAut}$ be its order. Fix a bijection $b\colon\mathrm{MAut}\to \{1,2,\ldots,r\}$.
\item[Step 2] Take $x\in\mathrm{MAut}$, multiplying it by each element in $\mathrm{MAut}$ and using $b$, associate to $x$ a permutation $\sigma_x = b\circ (\cdot x)\circ b^{-1}$ of $\{1,2,\ldots,r\}$.
\item[Step 3] Repeat Step~$2$ until the group generated by the $\sigma_x$'s has order~$r$.
\end{description}
These steps give a permutation group (namely a subgroup of the symmetric group $S_{r}$) that is isomorphic to $\mathrm{MAut}$ with an isomorphism given by the generators.
\end{Algorithm}

\subsection{Excluding non-modular automorphisms}\label{sec:non-MAut}

We now handle computationally the problem of Section \ref{sec:theory_modularity}, namely deciding whether all the automorphisms of a (geometrically connected) modular curve $X_H$ are modular. 

Since there are only finitely many modular automorphisms, the answer is ``no'' for curves of genus $g=0$ or $1$. In genus $g\ge 2$, we can embed $X_H$ in a projective space,  compute the  group  $\Aut(X_{H,\ol \QQ})$ and compare its size with the size of the group of modular automorphisms, computed with Algorithm \ref{alg:modaut}. This procedure gives an answer to our question, but it is not always practical: for hyperelliptic curves, plane quartics, and in general when we have embedding with small degree and in a projective space of small dimension, the computation of $\Aut(X_{H,\ol \QQ})$ is feasible, but in general it can become computationally expensive. 

We then look for practical methods to check whether all the automorphisms of a modular curve $X_H$ of genus $g\ge 2$ are modular, possibly without getting an answer. 
Theorem \ref{th:highgenus} gives a sufficient criterion to prove that all automorphisms are modular. Such criterion can be checked quite easily, indeed we have checked it on curves of level at most~$70$, and it is satisfied often in high genus, see Section \ref{sec: computational}, in particular Figure \ref{fig:criterion-by-genus}. 
Moreover, in many cases, it is feasible to compute the automorphisms of the curves, after reducing modulo a prime $\ell$ of good reduction. By Lemma \ref{lem:injaut}, our target group $G:= \Aut(X_{H,\ol \QQ})$ is a subgroup of $\Aut(X_{H, \ol\FF_{\ell}})$. More than that, it is a subgroup of $G_\ell:= \Aut(X_{H, \FF_{\ell^k}})$ where $k=k_\ell$ is the inertia degree of (the prime above) $\ell$ in the minimal number field where all the automorphisms of $X_H$ are defined. 
In particular, $\#G$ divides all possible cardinalities $\# G_\ell$ and the $\gcd$ of such cardinalities gives an upper bound on $\#G$, which we can again compare with the number of modular automorphisms. 

In the following result we give a (multiplicative) bound on the degrees~$k_\ell$, using Corollary \ref{cor:field_subabelian_varieties_J_H} and the decomposition over~$\QQ$ of the Jacobian~$J_H$ (see \cite{RSZ22}).

\begin{prop} 
\label{cor:lowgenus}
Let $H<\GLn$ be a subgroup with $\det(H) = (\ZZ/n\ZZ)^\times$ and $-\Id\in H$ and such that $X_H$ has genus $g \ge 2$. 
Let $F$ be the minimal number field such that all the automorphisms of $X_H$ are defined over~$F$. 
Let $\ell$ be a prime number not dividing $n$ and let $k$ be the inertia degree of a prime in $F$ above~$\ell$.

Consider the isogeny decomposition of $J_H$ over~$\QQ$.
\begin{equation}\label{eq:isog_decomp}
J_H \sim \prod_{h \in S} A_h^{m_h}
\end{equation} 
where $S$ is a suitable set of newforms, and, for each newform $h$ belonging to this set, $A_h$ is the abelian variety associated to $h$ and $m_h$ is a suitable non-negative integer. For each subset $S' \subset S$, we denote 
\[
\deg(S') := \sum_{h\in S'} m_h \dim A_h .
\]
\begin{enumerate}[label=\alph*)]
\item\label{item:compositum_quad_crit} If the subset $S_0\subset S$ of forms that are not CM and have trivial character satisfies $\deg(S_0) \ge \tfrac 12 g+1$, then $F$ is contained in a compositum of quadratic extensions of $\QQ$. In particular, $k$ divides~$D:=2$.
\item\label{item:field_zetan_crit} If the subset $S_1\subset S$ of forms that are not CM satisfies $\deg(S_1) \ge \tfrac 12 g+1$, then $F$ is contained in a compositum of quadratic extensions of $\QQ(\zeta_n)$. In particular, $k$ divides $D:=2 \mathrm{ord}_{(\ZZ/n\ZZ)^\times}(\ell)$.
\item\label{item:field_withCM_crit} Let $S_2$ be any subset of CM forms appearing in $S$ that satisfies $\deg(S_1  \cup S_2) \ge \tfrac 12 g+1$ and let $\{K_1, \ldots, K_s \}$ be the quadratic imaginary fields such that there is a form $f\in S_2$ having CM by~$K_i$; then $F$ is contained in the compositum of the ray class fields of $K_i$ associated to the ideals $n^2 \calO_{K_i}$ and a compositum of quadratic extensions of $\QQ(\zeta_n)$. In particular, $k$ divides
\[
D:= 2 n^4 \cdot \lcm_{1\le i \le s} (\# \mathrm{Cl}(K_i)) \cdot \lcm_{p|n}(p^2{-}1).
\]
\end{enumerate}
\end{prop}

\begin{proof}
By \cite[Lemma~4.1]{DLM22}, if we find a field~$L$ and an abelian variety $A$ isogenous to a subvariety of~$J_H$ satisfying $\dim A\ge \tfrac12 g+1$ and such that all the endomorphisms of $A$ are defined over~$L$, then $F \subset L$.

For Part~\ref{item:compositum_quad_crit}, we pick $A = \prod_{h\in S_0} A_h^{m_h}$, and the result follows by Corollary \ref{cor:field_subabelian_varieties_J_H}, Part~\ref{item:compositum_quad}.
Analogously, for Part~\ref{item:field_zetan_crit} we pick $A = \prod_{h\in S_1} A_h^{m_h}$, and the result follows by Corollary \ref{cor:field_subabelian_varieties_J_H} Part~\ref{item:field_zetan}. 

For Part~\ref{item:field_withCM_crit}, we pick $A = \prod_{h\in S_1 \cup S_2} A_h^{m_h}$. To prove that all the endomorphisms of $A$ are defined over a field as in the statement, we consider the decomposition 
\[
A = A_0 \times \ldots \times A_s,
\]
where $A_0 = \prod_{h\in S_1} A_h^{m_h}$ and, for $i \ge 1$, we define $A_i$ to be the product of $A_h^{m_h}$ for $h$ varying across the forms in $S_2$ with CM by $K_i$. Then, for $i\ne j $, the irreducible factors of $A_i$ are distinct from those of $A_j$, implying that 
\[
\End(A) = \End(A_0) \times \ldots \times \End(A_s).
\]
In particular, the minimal field $L$ where all the elements in $\End(A)$ are defined is the compositum of the minimal fields $L_i$, for $i \in \{0, \ldots, s\}$, where all the elements in $\End(A_i)$ are defined. So, we can bound~$L$, and hence~$F$, using parts~\ref{item:field_zetan} and~\ref{item:field_withCM} of Corollary \ref{cor:field_subabelian_varieties_J_H}.
Finally, to prove that $k$ divides~$D$, we prove that the exponent of the group $\Gal(L/\QQ)$ divides~$D$. It is enough to prove that the exponent of each Galois group $\Gal(L_i/\QQ)$ divides~$D/2$. For $L_0$ this is true because $2 \mathrm{ord}_{(\ZZ/n\ZZ)^\times} \ell $ divides~$D$. For $i\in \{1,\ldots, s\}$, we first observe that $\Gal(L_i/K_i)$ has index~$2$ in~$\Gal(L_i/\QQ)$, hence it is enough to prove that the exponent of~$\Gal(L_i/K_i)$  divides~$D$. Now, denoting by $\mathcal O_{K_i}$ the maximal order in~$K_i$, since $L_i$ is contained in the ray class field of $K_i$, with respect to  $n^2\mathcal O_{K_i}$, the exponent of $\Gal(L_i/K_i)$ divides the order of the ray class group $\mathrm{Cl}_{n^2\mathcal O_{K_i}} (\mathcal O_{K_i})$. The latter is an extension of  $\mathrm{Cl}(K_i)$ by $ (\calO_{K_i}/n^2\mathcal O_{K_i})^\times$ therefore its exponent divides the product of the order of $\mathrm{Cl}(K_i)$ and the exponent of $(\calO_{K_i}/n^2\mathcal O_{K_i})^\times$. Finally, considering the prime factorization $n^2 = \prod_{p\mid n}p^{e_p}$, we obtain the desired estimate by recalling that $(\calO_{K_i}/n^2\mathcal O_{K_i})^\times = \prod_{p\mid n}(\calO_{K_i}/p^{e_p}\mathcal O_{K_i})^\times$, and, depending on the splitting of~$p$, we can multiplicatively bound the order of each factor.
\end{proof}

\begin{rem}
Notice that in Part~\ref{item:field_withCM_crit} of the above Proposition \ref{cor:lowgenus}, we can simply choose $S_2$ equal to the set of all CM forms in $S$: it satisfies the condition since the sum is equal to~$g$. In particular, when $g\ge 2$, Part~\ref{item:field_withCM_crit} can always be applied to bound the inertia degree of the field of definition of the automorphisms of $X_H$, even though, when applicable, a stronger bound can be obtained from Part~\ref{item:compositum_quad_crit} or Part~\ref{item:field_zetan_crit} or by choosing a set $S_2$ smaller than the one of all CM forms.
\end{rem}

As a result of the discussion above, we present an algorithm that, in many cases, proves that all automorphisms of $X_H$ are modular without computing $\Aut(X_H)$. In the remaining cases, the algorithm either proves that some automorphisms are not modular or is inconclusive.

\begin{Algorithm}\label{alg:non-mod}
Given a subgroup $H<\GLn$ containing $-\Id$ and with $\det(H) = (\ZZ/n\ZZ)^\times$ either prove that all the automorphisms of $X_H$ are modular, or disprove it, or fail. In the input, a tolerance parameter $t \in \ZZ_{>0}$ must be given.
\begin{description}
\item[Step 0 (Genus Check)] If the genus of $X_H$ is $0$ or $1$, then output 
\[
\text{\emph{$X_H$ has non-modular automorphisms.}}
\]

\item[Step 1 (Decomposition of the Jacobian)] Compute modular forms $h$ and multiplicities $m_h$ giving  the isogeny decomposition 
\begin{equation}\label{eq: decomposition JH in algo}
J_H \sim \prod A_h^{m_h}
\end{equation}
of the Jacobian of $X_H$.

\item[Step 2 (Check condition in Theorem \ref{cor:CShighgen})] 
Compute 
\begin{itemize}
\item a lower bound $\widetilde{\rho} \le \textnormal{gonality}_{\ol{\QQ}}(X_H) $;
\item the smallest prime $\ell$ that does not divide $n$;
\item the set $S_1$ of forms $h$ in the decomposition \eqref{eq: decomposition JH in algo} that are not CM.
\end{itemize}
If
\begin{equation*}
\widetilde{\rho} \ge  2\ell +2 \quad \text{and} \quad \sum_{h\in S_1} m_h \dim A_h \ge \frac 12 \mathrm{genus}(X_H) +1,
\end{equation*}
then output 
\[
\text{\emph{Yes, all the automorphisms of $X_H$ are modular :)}} 
\]

\item[Step 3 (Modular automorphisms)]
Compute the cardinality $\mu$ of the group of modular automorphisms of $X_H$.

\item[Step 4 (Explicit model)]\label{step_eqs} Check if  $X_H$ is hyperelliptic and compute equations for a model of $X_H$ over~$\QQ$.

If $X_H$ is hyperelliptic and $\mu$ is odd, then output 
\[
\text{\emph{$X_H$ has non-modular automorphisms.}}
\]

\item[Step 5 (Automorphisms of the reductions)] Compute all primes of bad reduction for the model computed above and consequently the smallest $t$ primes $p_1< \ldots < p_t$ of good reduction. For $j\in\{1,\ldots,t\}$ do as follows:
\begin{description}
\item[Step 5.a] Compute the integer $D$ as in Proposition \ref{cor:lowgenus}: if the condition in Part~\ref{item:compositum_quad_crit} is satisfied, set $D=2$; otherwise, if Part~\ref{item:field_zetan_crit} is satisfied, set $D=2 \mathrm{ord}_{(\ZZ/n\ZZ)^\times} (p_j)$; otherwise, choose $S_2$ as in Part~\ref{item:field_withCM_crit} and set $D$ as in the statement.

\item[Step 5.b] Reduce the equations in Step~$4$ modulo~$p_j$ and compute the cardinality $\nu_j$ of the automorphism group of $(X_H)_{\FF_{p_j^D}}$.

\end{description}
\item[Step 6 (Bound from the automorphisms of the reductions)] If $\gcd \{\nu_{1},\ldots, \nu_t\}=\mu$, then output 
\[
\text{\emph{Yes, all the automorphisms of $X_H$ are modular :)}} 
\]

\item[Step 7]\label{stepfail} Output 
\[
\text{\emph{Fail :( }}
\]
\end{description}
\end{Algorithm}

\begin{rem}
The tolerance parameter $t$ in Algorithm \ref{alg:non-mod} is the number of attempts in finite characteristic. If this number is large enough and the algorithm terminates with a failure in Step~$7$, one might reasonably conjecture the existence of non-modular automorphisms; the higher~$t$, the more likely the existence of non-modular automorphisms. Moreover, $\gcd\{\nu_{1},\ldots, \nu_t\}$ is a (multiplicative) upper bound for $|\Aut(X_H)|$ giving an upper bound for the number of non-modular automorphisms, if any.
\end{rem}

\begin{rem}[Computability of each step]
In Step~$1$ we can use the algorithm described in \cite{RSZ22} to find the newforms associated to the isogeny decomposition of the Jacobian of~$X_H$.

In Step~$2$, the forms that are not CM can be identified using \cite{donnellyVoightlmfdb} or \cite[Section 11]{LMFDBpaper}, while, for $\tilde \rho$, we can take the largest lower bound among the one given in \cite{Abr96} and those computed using \cite{Poonen_gonality2007} for the smallest primes $\ell_1<\ldots<\ell_t$ not dividing~$n$.

In Step~$3$ we can use Algorithm \ref{alg:modaut} to compute the cardinality of the group of modular automorphisms of~$X_H$. 

In Step~$4$ we can use \cite{Zyw21} (see also the files \texttt{\detokenize{GL2GroupTheory.m}} and \texttt{\detokenize{ModularCurves.m}} in \cite{Zywcode}) to find a basis of differentials and from it, compute the rest: one can check if the curve is hyperelliptic by checking if the space of quadratic relations among differentials has dimension $(g-1)(g-2)/2$, with $g$ the genus of the curve; 
if the curve is hyperelliptic, one can compute a hyperelliptic model as in \cite[Chapter 4]{galbraith1996equations}, otherwise we can compute the canonical embedding as in \cite{Mer18} or \cite{Zyw21}.

The other computations can be done by internal functions of Magma.
\end{rem}

\section{Computational results}
\label{sec: computational}

The LMFDB database \cite{lmfdb}, in its beta version \url{https://beta.lmfdb.org/}, collects, among many other things, data on every geometrically connected modular curve of level at most~$70$, i.e., on every curve $X_H$ with $H<\GLn$ and $n\le 70$, such that $\det(H)=(\ZZ/n\ZZ)^\times$. We are only interested in subgroups~$H$ containing~$-\Id$, since it does not change the coarse moduli space~$X_H$. For each such curve, the database records the level, the index, the genus, rigorous bounds for the gonality, and the isogeny decomposition of the Jacobian~$J_H$ into newform pieces together with their multiplicities. We used these data to check the criterion of Theorem \ref{cor:CShighgen} and to compute the group of modular automorphisms, following the algorithm of Section \ref{sec:MAut}. 
We also look for curves admitting exceptional automorphisms.

The data produced can be found on the GitHub repository \cite{script_MAut}.

\subsection{Checking a sufficient criterion for modularity}\label{sec:criterion-stats}

We checked the two conditions of Theorem \ref{cor:CShighgen} for all the geometrically connected modular curves~$X_H$ of genus at least~$2$ and level at most~$70$, which are $797\,589$ curves. 

Out of the 
curves considered, the criterion is satisfied for $507\,967$ of them, a success rate of $63.69\%$, thus proving that all automorphisms of such curves are modular. 
As expected, this success rate is far from uniform in the genus: as Figure \ref{fig:criterion-by-genus} shows, it is $0\%$ for small genus, and increases to $100\%$ once the genus is large enough. More precisely, we have the following: 
for the $36\,921$ curves with genus between $2$ and $9$,  
the criterion is never satisfied; for the $560\,498$ curves with 
genus between $10$ and $209$, the success rate is a more moderate $54.91\%$; for the $200\,170$ curves with  
genus at least~$210$, the criterion is always satisfied. The data represented in Figure \ref{fig:criterion-by-genus} are stored in the file \texttt{\detokenize{results70.xlsx}} in the GitHub repository \cite{script_MAut}.

\begin{figure}[htb]
\centering
\includegraphics[width=0.92\textwidth]{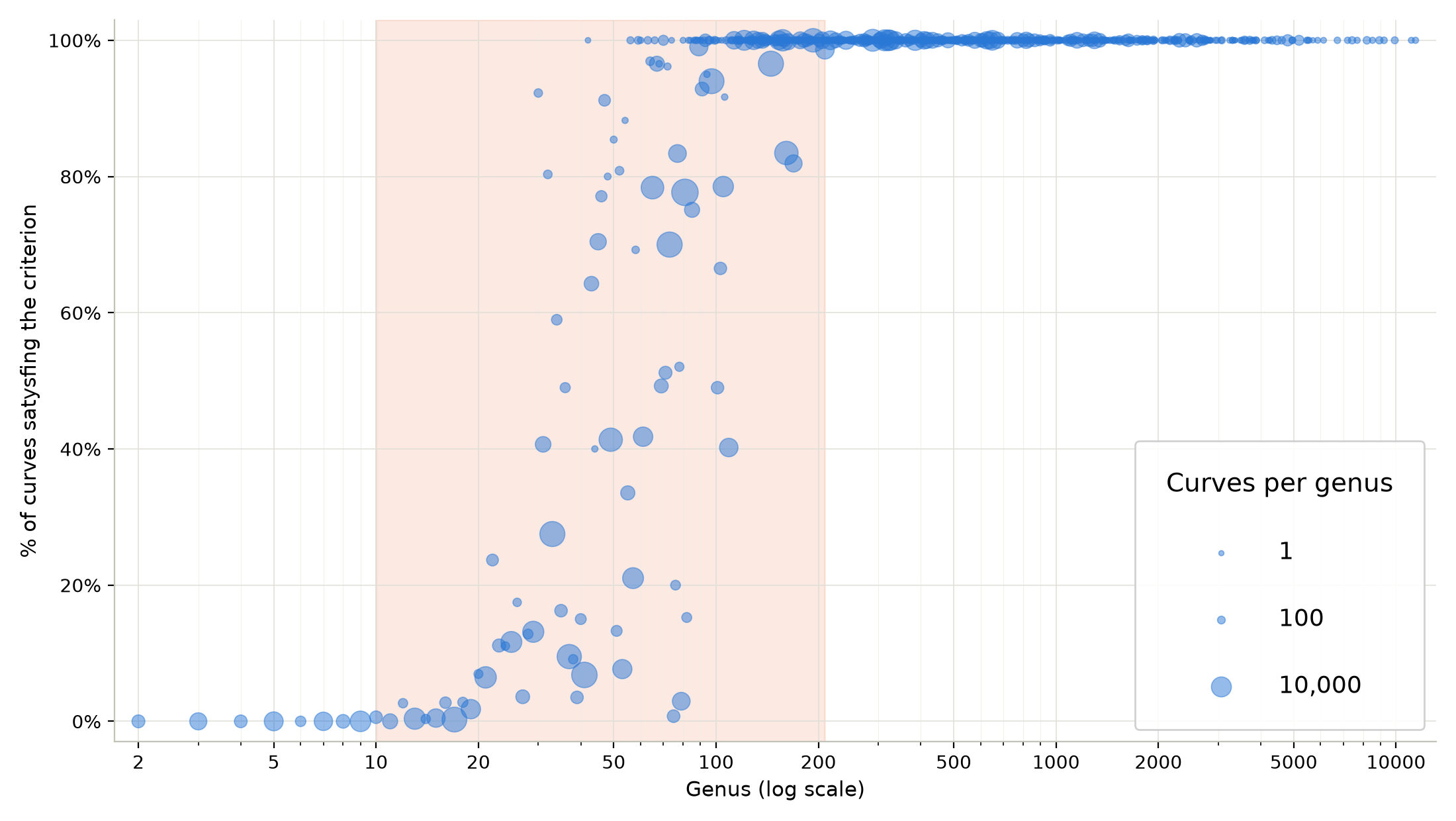}
\caption{Percentage of geometrically connected modular curves of level at most $70$ satisfying the criterion of Theorem \ref{cor:CShighgen}: for every genus $g \ge 2$ occurring among such curves, the height of the corresponding point is the percentage of the curves of genus~$g$ for which the criterion holds, and the area of the point grows with the number of such curves.
\\ 
The shaded band marks the range $10\le g\le 209$: outside of it the criterion is either never satisfied ($g\le 9$) or always satisfied ($g\ge 210$, among the curves of level at most~$70$ considered here).}
\label{fig:criterion-by-genus}
\end{figure}

\subsection{Computing modular automorphisms}

We computed the group of modular automorphisms for all the $56\,748$ geometrically connected modular curves of level at most~$39$ and genus at least~$2$. See the file \texttt{\detokenize{MAut_results_up_to_level_39.txt}} in the GitHub repository \cite{script_MAut} for the output of this computation, namely, for each curve~$X_H$, the Magma code of the abstract group structure of $\mathrm{MAut}(X_H)$ and representatives of generators in the form of matrices with coprime integer entries.

We show in Table \ref{tab:more_freq} the abstract group structures occurring more frequently.

\begin{table}[H]
\centering
\begin{tabular}{lrr}
\hline
\textbf{Group description} & \textbf{Number of curves} & \textbf{Percentage} \\
\hline
$C_2 \times D_4$                                  & 4477 & 7.89\\
$C_2 \times (C_2^2 \rtimes C_4)$                  & 3296 & 5.81\\
$C_2^2$                                           & 2854 & 5.03\\
$C_2^3$                                           & 2095 & 3.69\\
$C_2^2 \rtimes C_4$                               & 1768 & 3.12\\
$C_2 \times (C_2^2 . D_4)$                        & 1552 & 2.73\\
$C_2^2 \wr C_2$                                   & 1188 & 2.09\\
$D_6$                                             & 1162 & 2.05\\
$C_2^4 . C_2^3$                                   & 1152 & 2.03\\
$A_4 \times (C_2^2 \rtimes C_4)$                  & 1088 & 1.92\\
$D_4$                                             & 1019 & 1.80\\
$C_2^2 \times S_3$                                & 1009 & 1.78\\
$\mathrm{GL}_2(\mathbb Z/4\mathbb Z)$             &  899 & 1.58\\
$C_2 \times S_4$                                  &  847 & 1.49\\
$C_2^4 . (C_2 \times C_{12})$                     &  640 & 1.13\\
$(C_2 . C_4^2) . D_6$                             &  612 & 1.08\\
$(C_2^2 \times C_4) \rtimes C_2$                  &  608 & 1.07\\
$C_2$                                             &  601 & 1.06\\
$C_2 \times (D_4 \rtimes C_2)$                    &  576 & 1.01\\
$D_4 \times A_4$                                  &  512 & 0.90\\
$(C_2^3 . (A_4 \rtimes C_4)) . C_2$               &  512 & 0.90\\
$C_2^2 \times D_4$                                &  512 & 0.90\\
\hline
\textbf{Total}                                    & \textbf{28979} & \textbf{51.06} \\ 
\hline
\end{tabular}
\caption{Most frequent structural descriptions of modular automorphism groups. We use $A.B$ to denote an extension of $A$ by $B$ (not necessarily split),
while $\times$, $\rtimes$, and $\wr$ denote direct, semidirect, and wreath
products, respectively.}
\label{tab:more_freq}
\end{table}

It turns out that only $72$ curves among the curves we analyzed ($0.127\%$) have trivial modular automorphism group.

By contrast, we found more ``maximal'' curves for what concerns the number of automorphisms. Specifically, we found $214$ curves $X_H$ having genus $g\ge 2$ such that the cardinality of $\mathrm{MAut}(X_H)$ is equal to the maximum number of automorphisms for a curve over~$\CC$ with that given genus~$g$. In particular, for all such curves, one deduces that all their automorphisms are necessarily modular. 
For example, for the five curves \href{https://beta.lmfdb.org/ModularCurve/Q/7.168.3.b.1}{7.168.3.b.1}, \href{https://beta.lmfdb.org/ModularCurve/Q/7.168.3.a.1}{7.168.3.a.1}, \href{https://beta.lmfdb.org/ModularCurve/Q/21.168.3.a.1}{21.168.3.a.1}, \href{https://beta.lmfdb.org/ModularCurve/Q/28.168.3.a.1}{28.168.3.a.1}, \href{https://beta.lmfdb.org/ModularCurve/Q/35.168.3.a.1}{35.168.3.a.1} (using the LMFDB label), the group of modular automorphisms hits the Hurwitz bound $84(g-1)$.
For example, the first curve, which is a model over $\QQ$ of one of the components of $X(7)_{\ol\QQ}$ since we have $\Gamma_H = \Gamma(7)$, has genus~$3$ and is isomorphic to the Klein quartic
\[
x^3y + xz^3 + y^3z = 0 ,
\]
and the embeddings 
\[
\mathrm{PSL}_2(\FF_7) = \mathrm{PSL}_2(\ZZ)/\Gamma(7) 
\hookrightarrow \mathrm{MAut}(X_H)
\hookrightarrow \mathrm{Aut}(X_H),
\]
already give a subgroup of the automorphism group of cardinality $168 = 84(g-1)$. 
Moreover, for many values of $g$ the Hurwitz bound cannot be achieved (for example, there are only 32 values $g \in  [2, 11905]$ for which Hurwitz bound can be hit, see the main theorem in \cite[Section 6]{Conder_87} for the list). 
In \cite{ConderAMS} Conder has computed, for each $g \in [2,301]$, the maximum number of automorphisms of a complex curve of genus~$g$, see Appendix~1 of loc. cit. or \cite{Conderlist} for the complete list. Comparing these data with our computations, we found the already cited ``maximal'' curves which we count, depending on the genus, in Table \ref{tab:max}.

\begin{table}[H]
\centering
\begin{tabular}{ccccc}
\hline
\textbf{Genus} &
\textbf{Number of maximal curves} &
\textbf{Total in genus} &
\textbf{Percentage} &
$\boldsymbol{|\mathrm{MAut}|=\max|\mathrm{Aut}|}$ \\
\hline
2   & 12  & 865  & 1.39  & 48   \\
3   & 5   & 2300 & 0.22  & 168  \\
4   & 6   & 782  & 0.77  & 120  \\
5   & 108 & 3252 & 3.32  & 192  \\
13  & 44  & 2881 & 1.53  & 360  \\
81  & 26  & 879  & 2.96  & 1536 \\
141 & 8   & 27   & 29.63 & 2184 \\
196 & 5   & 7    & 71.43 & 3420 \\
\hline
\end{tabular}
\caption{Genera for which the modular automorphism group attains the maximum possible order of the automorphism group. The percentage is computed with respect to all curves of the corresponding genus in the database of levels at most $39$.}
\label{tab:max}
\end{table}

Continuing on the ``large group of modular automorphisms'' topic, we found more curves for which one can prove modularity of all automorphisms just using the abundance of modular automorphisms.

\begin{thm}
The $972$ modular curves $X_H$ listed in Table \ref{tab: many modaut} in the Appendix  satisfy 
\begin{equation}\label{eq: many modaut}
|\mathrm{MAut}(X_H)|>\frac{1}{2}\max_{\operatorname{genus}(X)=\operatorname{genus}(X_H)}|\mathrm{Aut}(X)|.
\end{equation}
In particular, all their automorphisms are modular.
\end{thm}
\begin{proof}
It follows from the file \texttt{\detokenize{MAut_results_up_to_level_39.txt}} in \cite{script_MAut} and the comparison with the values of the maximum in \cite[Appendix~1]{ConderAMS} or \cite{Conderlist} for a given genus.
\end{proof}

Finally, we look at some particular cases. 
In Table \ref{tab:exotic} we show a selection of some large and exotic groups we encountered.
\begin{table}[H]
\centering
\begin{tabular}{lrrr}
\hline
\textbf{Group description} & \textbf{Order} & \textbf{Number of curves} & \textbf{Percentage} \\
\hline
$C_2^3.(A_4 \rtimes C_4).C_2$ & 768 & 512 & 0.90\\
$(C_4 \times Q_8).C_{12}.C_2$ & 768 & 224 & 0.39\\
$C_2^2.\mathrm{SD}_{16}.C_2$ & 128 & 223 & 0.39\\
$C_2^3.(Q_8 \times A_4).C_6$ & 4608 & 165 & 0.29\\
$Q_8.(C_3 \times (C_2^2.D_4))$ & 768 & 161 & 0.28\\
$(C_2 \times (D_4 \rtimes C_2)).C_{12}.C_2$ & 768 & 86 & 0.15\\
$(C_2^2 \times C_4).(C_2 \wr C_4)$ & 1024 & 70 & 0.12\\
$(C_4^2 \times C_8).C_4^2.S_3$ & 12288 & 32 & 0.06\\
$C_6^3.(S_3 \times A_4)$ & 15552 & 30 & 0.05\\
$(A_4 \rtimes C_4).\mathrm{PSL}_2(7)$ & 8064 & 30 & 0.05\\
$(C_2 \times \mathrm{OD}_{32}).C_8.C_2$ & 1024 & 22 & 0.04\\
$\mathrm{PSL}_2(19) \times S_3$ & 20520 & 11 & 0.02\\
$\mathrm{PSL}_2(37)$ & 25308 & 10 & 0.02\\
$\mathrm{SL}_2(3).\mathrm{PSL}_2(13)$ & 26208 & 6 & 0.01\\
\hline
\end{tabular}
\caption{Some examples of particularly rich group structures occurring among the modular automorphism groups. Here $C_n$, $D_n$, $Q_n$, $\mathrm{SD}_n$, $\mathrm{OD}_n$, $S_n$, and $A_n$ denote, respectively, the cyclic, dihedral, generalized quaternion, semidihedral, other-dihedral, symmetric, and alternating groups. The notation $\mathrm{SL}_n(q)$ and $\mathrm{PSL}_n(q)$ has its usual meaning. We use $\times$, $\rtimes$, $\wr$, and $.$ for direct products, semidirect products, wreath products, and extensions (not necessarily split), respectively.}
\label{tab:exotic}
\end{table}

\subsection{Exceptional candidates up to some genus}

We looked for curves with exceptional automorphisms among hyperelliptic curves. Indeed, many of the known examples of modular curves with exceptional automorphisms known in the literature are hyperelliptic. Moreover, for such curves, one cannot apply Theorem \ref{cor:CShighgen} since, by definition their gonality is equal to~$2$. On the other hand, using their  hyperelliptic model, we can compute their automorphism group efficiently, e.g. using Magma's internal function \texttt{AutomorphismsOfHyperellipticCurve}, and compare it with the group of modular automorphisms.  
Applying this procedure to the modular curves of level at most $39$ that admit a hyperelliptic model in the LMFDB, we obtain the following.

\begin{thm}\label{thm: non modular cases}
The $139$  hyperelliptic modular curves listed in Table \ref{tab: non-modular aut} admit a non-modular automorphism.
\end{thm}

Among the curves we found, we recovered some known examples: the curve $X_\mathrm{s}^+(11)$ labeled by LMFDB as \href{https://beta.lmfdb.org/ModularCurve/Q/11.66.2.a.1}{11.66.2.a.1}; the curve $X_\mathrm{ns}^+(16)$ labeled by LMFDB as \href{https://beta.lmfdb.org/ModularCurve/Q/16.64.2.a.1}{16.64.2.a.1}; and the curve $X_0(37)$ labeled by LMFDB as \href{https://beta.lmfdb.org/ModularCurve/Q/37.38.2.a.1}{37.38.2.a.1}.
We also find the following examples.

\begin{ese}
The curve $X_H$ labeled   \href{https://beta.lmfdb.org/ModularCurve/Q/32.48.2.b.2}{32.48.2.b.2} in the LMFDB has $|\mathrm{Aut}(X_H)|=48$ and $|\mathrm{MAut}(X_H)|=16$. In particular, there is at least one non-modular automorphism of order~$3$. The automorphisms defined over the rationals are all modular: the curve admits the following model over $\QQ$
\[
Y^2=-2X^5Z -2XZ^5,
\]
and the group  $\mathrm{Aut}_\QQ(X_H)$ has order~$4$, with generators
\[
g_1\colon
\begin{cases}
X \longmapsto X,\\
Y \longmapsto -Y,\\
Z \longmapsto Z,
\end{cases}
\qquad
g_2\colon
\begin{cases}
X \longmapsto Z,\\
Y \longmapsto Y,\\
Z \longmapsto X,
\end{cases}
\]
which are both modular since they preserve the set of cusps (which is the set of points with $Y=0$)
and the set of elliptic points (which is empty).
\end{ese}

\printbibliography

\section*{Appendix}\label{appendix: almost maximal}

In \Cref{tab: many modaut} we list all the modular curves $X_H$ with level $n\le 39$ and genus  $g \le 301$ with ``many'' modular automorphisms, i.e., satisfying \Cref{eq: many modaut}. Notice that this condition implies that all the automorphisms are modular. The curves in bold are the ones for which we could already prove that all the automorphisms are modular  because they satisfy the hypotheses of \Cref{th:highgenus}.

\begin{longtable}{c c p{0.8\textwidth}} 
\caption{Curves with many modular automorphisms (see \eqref{eq: many modaut})}
\label{tab: many modaut}\\                        
\hline
\textbf{genus} & $\mathbf{|\mathrm{MAut}(X_H)|}$ &  \textbf{Curves (LMFDB label)} \\
\hline
\\
\endhead                                                       
\hline
\endfoot            
\\
\hline
\\
2 & 48 &  8.48.2.a.1,\ 16.48.2.a.1,\ 16.48.2.b.1,\ 16.48.2.e.1,\ 16.48.2.f.1, 16.48.2.l.1,\ 16.48.2.l.2,\ 16.48.2.m.1,\ 16.48.2.m.2,\ 24.48.2.a.1, 32.48.2.a.1,\ 32.48.2.a.2  \\ 
3 & 96 &  8.96.3.a.1,\ 8.96.3.b.1,\ 8.96.3.e.1,\ 8.96.3.g.1,\ 8.96.3.g.2, 8.96.3.k.1,\ 8.96.3.l.1,\ 8.96.3.m.1,\ 8.96.3.n.1,\ 8.96.3.o.1, 8.96.3.p.1,\ 16.96.3.cq.1,\ 16.96.3.cw.1,\ 16.96.3.cy.1,\ 24.96.3.a.1, 24.96.3.b.1,\ 24.96.3.c.1,\ 24.96.3.dc.1,\ 24.96.3.dd.1,\ 24.96.3.eu.1, 24.96.3.ev.1,\ 24.96.3.ew.1,\ 24.96.3.ex.1,\ 24.96.3.ic.1,\ 24.96.3.id.1, 24.96.3.ie.1,\ 24.96.3.if.1,\ 24.96.3.ig.1,\ 24.96.3.ih.1,\ 24.96.3.ii.1, 24.96.3.ij.1,\ 24.96.3.ik.1,\ 24.96.3.il.1,\ 24.96.3.im.1,\ 24.96.3.in.1, 24.96.3.io.1,\ 24.96.3.ip.1,\ 24.96.3.iq.1,\ 24.96.3.ir.1,\ 24.96.3.o.1, 24.96.3.o.2,\ 24.96.3.q.1,\ 24.96.3.q.2,\ 24.96.3.r.1,\ 24.96.3.r.2, 24.96.3.r.3,\ 32.96.3.a.1,\ 32.96.3.a.2,\ 32.96.3.d.1,\ 32.96.3.d.2, 32.96.3.e.1,\ 32.96.3.g.1,\ 32.96.3.j.1,\ 32.96.3.j.2,\ 32.96.3.k.1, 32.96.3.m.1,\ 32.96.3.o.1,\ 32.96.3.p.1 \\ 
3 & 168 &  7.168.3.a.1,\ 7.168.3.b.1,\ 21.168.3.a.1,\ 28.168.3.a.1,\ 35.168.3.a.1  \\ 
4 & 72 &  12.72.4.a.1,\ 12.72.4.b.1,\ 12.72.4.ba.1,\ 12.72.4.c.1,\ 12.72.4.d.1, 12.72.4.e.1,\ 12.72.4.p.1,\ 12.72.4.q.1,\ 12.72.4.v.1,\ 12.72.4.w.1, 12.72.4.z.1,\ 24.72.4.a.1,\ 24.72.4.b.1,\ 24.72.4.c.1,\ 24.72.4.cu.1, 24.72.4.cv.1,\ 24.72.4.cw.1,\ 24.72.4.cx.1,\ 24.72.4.d.1,\ 24.72.4.dw.1, 24.72.4.dx.1,\ 24.72.4.e.1,\ 24.72.4.ea.1,\ 24.72.4.eb.1,\ 24.72.4.ec.1, 24.72.4.ed.1,\ 24.72.4.f.1,\ 24.72.4.g.1,\ 24.72.4.h.1,\ 24.72.4.i.1, 24.72.4.j.1,\ 24.72.4.jq.1,\ 24.72.4.jr.1,\ 24.72.4.ju.1,\ 24.72.4.jv.1, 24.72.4.jw.1,\ 24.72.4.jx.1,\ 24.72.4.ko.1,\ 24.72.4.kp.1,\ 24.72.4.ks.1, 24.72.4.kt.1,\ 24.72.4.ku.1,\ 24.72.4.kv.1,\ 36.72.4.g.1  \\ 
4 & 120 & 10.180.4.a.1,\ 10.180.4.b.1,\ 20.180.4.a.1,\ 20.180.4.b.1,\ 30.180.4.a.1, 30.180.4.b.1  \\ 
5 & 120 &  10.120.5.c.1,\ 10.120.5.d.1,\ 10.120.5.e.1,\ 10.120.5.h.1,\ 20.120.5.bl.1, 20.120.5.bm.1,\ 20.120.5.bn.1,\ 20.120.5.o.1,\ 20.120.5.p.1,\ 20.120.5.q.1, 20.120.5.r.1,\ 20.120.5.s.1,\ 20.120.5.t.1,\ 20.120.5.u.1,\ 20.120.5.v.1, 20.120.5.w.1,\ 30.120.5.c.1,\ 30.120.5.d.1,\ 30.120.5.e.1,\ 30.120.5.f.1, 30.120.5.i.1,\ 30.120.5.l.1,\ 30.120.5.m.1,\ 30.120.5.n.1,\ 30.120.5.o.1, 30.120.5.s.1,\ 30.120.5.u.1,\ 30.120.5.x.1  \\ 
5 & 192 &  8.192.5.a.1,\ 8.192.5.b.1,\ 8.192.5.b.2,\ 8.192.5.b.3,\ 8.192.5.c.1, 8.192.5.d.1,\ 8.192.5.d.2,\ 8.192.5.d.3,\ 16.192.5.bq.1,\ 16.192.5.br.1, 16.192.5.bz.1,\ 16.192.5.bz.2,\ 16.192.5.cf.1,\ 16.192.5.cg.1,\ 16.192.5.cj.1, 16.192.5.ck.1,\ 24.192.5.ba.1,\ 24.192.5.ba.2,\ 24.192.5.ba.3,\ 24.192.5.ba.4, 24.192.5.bb.1,\ 24.192.5.bc.1,\ 24.192.5.bd.1,\ 24.192.5.bd.2,\ 24.192.5.be.1, 24.192.5.bf.1,\ 24.192.5.bg.1,\ 24.192.5.bg.2,\ 24.192.5.bh.1,\ 24.192.5.bh.2, 24.192.5.bi.1,\ 24.192.5.bi.2,\ 24.192.5.bj.1,\ 24.192.5.bj.2,\ 24.192.5.bj.3, 24.192.5.bj.4,\ 24.192.5.i.1,\ 24.192.5.j.1,\ 24.192.5.k.1,\ 24.192.5.k.2, 24.192.5.l.1,\ 24.192.5.l.2,\ 24.192.5.m.1,\ 24.192.5.m.2,\ 24.192.5.n.1, 24.192.5.n.2,\ 24.192.5.n.3,\ 24.192.5.n.4,\ 24.192.5.o.1,\ 24.192.5.p.1, 24.192.5.p.2,\ 24.192.5.q.1,\ 24.192.5.q.2,\ 24.192.5.r.1,\ 24.192.5.r.2, 24.192.5.r.3,\ 24.192.5.s.1,\ 24.192.5.s.2,\ 24.192.5.s.3,\ 24.192.5.s.4, 24.192.5.t.1,\ 24.192.5.u.1,\ 24.192.5.v.1,\ 24.192.5.v.2,\ 24.192.5.w.1, 24.192.5.x.1,\ 24.192.5.x.2,\ 24.192.5.y.1,\ 24.192.5.y.2,\ 24.192.5.z.1, 24.192.5.z.2,\ 24.192.5.z.3,\ 32.192.5.ba.1,\ 32.192.5.ba.2,\ 32.192.5.bb.1, 32.192.5.bb.2,\ 32.192.5.c.1,\ 32.192.5.c.2,\ 32.192.5.c.3,\ 32.192.5.c.4, 32.192.5.d.1,\ 32.192.5.d.2,\ 32.192.5.d.3,\ 32.192.5.d.4,\ 32.192.5.e.1, 32.192.5.e.2,\ 32.192.5.e.3,\ 32.192.5.e.4,\ 32.192.5.f.1,\ 32.192.5.f.2, 32.192.5.g.1,\ 32.192.5.g.2,\ 32.192.5.j.1,\ 32.192.5.j.2,\ 32.192.5.k.1, 32.192.5.k.2,\ 32.192.5.v.1,\ 32.192.5.v.2,\ 32.192.5.v.3,\ 32.192.5.v.4, 32.192.5.x.1,\ 32.192.5.x.2,\ 32.192.5.x.3,\ 32.192.5.x.4,\ 32.192.5.y.1, 32.192.5.y.2,\ 32.192.5.z.1,\ 32.192.5.z.2  \\ 
9 & 192 &  24.288.9.a.1,\ 24.288.9.b.1,\ 24.288.9.b.2,\ 24.288.9.ba.2,\ 24.288.9.bb.1, 24.288.9.bb.2,\ 24.288.9.be.1,\ 24.288.9.bf.1,\ 24.288.9.bf.2,\ 24.288.9.bg.1, 24.288.9.bg.2,\ 24.288.9.bh.1,\ 24.288.9.bi.1,\ 24.288.9.bi.2,\ 24.288.9.bj.1, 24.288.9.bj.2,\ 24.288.9.bk.1,\ 24.288.9.bk.2,\ 24.288.9.c.1,\ 24.288.9.c.2, 24.288.9.d.1,\ 24.288.9.d.2,\ 24.288.9.e.1,\ 24.288.9.e.2,\ 24.288.9.f.1, 24.288.9.f.2,\ 24.288.9.g.1,\ 24.288.9.g.2,\ 24.288.9.h.1,\ 24.288.9.i.1, 24.288.9.i.2,\ 24.288.9.j.1,\ 24.288.9.j.2,\ 24.288.9.k.1,\ 24.288.9.l.1, 24.288.9.l.2,\ 24.288.9.m.1,\ 24.288.9.n.1,\ 24.288.9.n.2,\ 24.288.9.o.1, 24.288.9.p.1,\ 24.288.9.p.2,\ 24.288.9.q.1,\ 24.288.9.r.1,\ 24.288.9.r.2, 24.288.9.s.1,\ 24.288.9.t.1,\ 24.288.9.u.2,\ 24.288.9.v.1,\ 24.288.9.x.1  \\ 
10 & 324 & {\bf 9.324.10.b.1},\  9.324.10.a.1,\ 36.324.10.a.1  \\ 
13 & 288 &  12.288.13.a.1,\ 12.288.13.b.1,\ 12.288.13.ba.1,\ 12.288.13.bb.1,\ 12.288.13.bc.1, 12.288.13.bd.1,\ 12.288.13.c.1,\ 12.288.13.d.1,\ 12.288.13.g.1,\ 12.288.13.h.1, 12.288.13.i.1,\ 12.288.13.j.1,\ 12.288.13.k.1,\ 12.288.13.l.1,\ 12.288.13.r.1, 12.288.13.t.1,\ 12.288.13.u.1,\ 12.288.13.v.1,\ 12.288.13.w.1,\ 12.288.13.x.1, 12.288.13.y.1,\ 12.288.13.z.1,\ 24.288.13.a.1,\ 24.288.13.b.1,\ 24.288.13.c.1, 24.288.13.ci.1,\ 24.288.13.cj.1,\ 24.288.13.ck.1,\ 24.288.13.cl.1,\ 24.288.13.cm.1, 24.288.13.cn.1,\ 24.288.13.co.1,\ 24.288.13.cp.1,\ 24.288.13.cq.1,\ 24.288.13.cr.1, 24.288.13.cs.1,\ 24.288.13.ct.1,\ 24.288.13.cu.1,\ 24.288.13.cv.1,\ 24.288.13.cw.1, 24.288.13.cx.1,\ 24.288.13.cy.1,\ 24.288.13.cz.1,\ 24.288.13.d.1,\ 24.288.13.da.1, 24.288.13.db.1,\ 24.288.13.dc.1,\ 24.288.13.dd.1,\ 24.288.13.de.1,\ 24.288.13.df.1, 24.288.13.dg.1,\ 24.288.13.dh.1,\ 24.288.13.di.1,\ 24.288.13.dj.1,\ 24.288.13.dk.1, 24.288.13.dl.1,\ 24.288.13.dm.1,\ 24.288.13.dn.1,\ 24.288.13.do.1,\ 24.288.13.dp.1, 24.288.13.e.1,\ 24.288.13.f.1,\ 24.288.13.fj.1,\ 24.288.13.fk.1,\ 24.288.13.fl.1, 24.288.13.fm.1,\ 24.288.13.fq.1,\ 24.288.13.fr.1,\ 24.288.13.fs.1,\ 24.288.13.ft.1, 24.288.13.fz.1,\ 24.288.13.g.1,\ 24.288.13.ga.1,\ 24.288.13.gb.1,\ 24.288.13.gc.1, 24.288.13.gg.1,\ 24.288.13.gh.1,\ 24.288.13.gi.1,\ 24.288.13.gj.1,\ 24.288.13.h.1, 24.288.13.hz.1,\ 24.288.13.i.1,\ 24.288.13.id.1,\ 24.288.13.ie.1,\ 24.288.13.if.1, 24.288.13.ig.1,\ 24.288.13.ih.1,\ 24.288.13.ik.1,\ 24.288.13.il.1,\ 24.288.13.im.1, 24.288.13.in.1,\ 24.288.13.io.1,\ 24.288.13.ip.1,\ 24.288.13.iq.1,\ 24.288.13.ir.1, 24.288.13.j.1,\ 24.288.13.jq.1,\ 24.288.13.jr.1,\ 24.288.13.js.1,\ 24.288.13.jt.1, 24.288.13.ju.1,\ 24.288.13.jv.1,\ 24.288.13.jw.1,\ 24.288.13.jx.1,\ 24.288.13.jy.1, 24.288.13.jz.1,\ 24.288.13.k.1,\ 24.288.13.ka.1,\ 24.288.13.kb.1,\ 24.288.13.kc.1, 24.288.13.kd.1,\ 24.288.13.l.1,\ 24.288.13.m.1,\ 24.288.13.n.1,\ 24.288.13.o.1, 24.288.13.p.1,\ 24.288.13.qk.1,\ 24.288.13.ql.1,\ 24.288.13.qm.1,\ 24.288.13.qn.1, 24.288.13.qo.1,\ 24.288.13.qp.1,\ 24.288.13.qq.1,\ 24.288.13.qr.1,\ 24.288.13.qs.1, 24.288.13.qt.1,\ 24.288.13.qu.1,\ 24.288.13.qv.1,\ 24.288.13.qw.1,\ 24.288.13.qx.1, 24.288.13.rg.1,\ 24.288.13.rh.1,\ 24.288.13.ri.1,\ 24.288.13.rj.1,\ 24.288.13.rk.1, 24.288.13.rl.1,\ 24.288.13.rm.1,\ 24.288.13.rn.1,\ 24.288.13.ro.1,\ 24.288.13.rp.1, 24.288.13.rq.1,\ 24.288.13.rr.1,\ 24.288.13.rs.1,\ 24.288.13.rt.1,\ 24.288.13.ru.1, 24.288.13.rv.1,\ 24.288.13.rw.1,\ 24.288.13.rx.1,\ 24.288.13.ry.1,\ 24.288.13.rz.1, 24.288.13.sa.1,\ 24.288.13.sb.1,\ 24.288.13.sc.1,\ 24.288.13.sd.1,\ 24.288.13.se.1, 24.288.13.sf.1,\ 24.288.13.sg.1,\ 24.288.13.sh.1,\ 24.288.13.sy.1,\ 24.288.13.sz.1, 24.288.13.ta.1,\ 24.288.13.tb.1,\ 24.288.13.tc.1,\ 24.288.13.td.1,\ 24.288.13.te.1, 24.288.13.tf.1,\ 24.288.13.tg.1,\ 24.288.13.th.1,\ 24.288.13.ti.1,\ 24.288.13.tj.1, 24.288.13.tk.1,\ 24.288.13.tl.1,\ 36.288.13.a.1,\ 36.288.13.b.1,\ 36.288.13.c.1, 36.288.13.d.1,\ 36.288.13.e.1,\ 36.288.13.m.1,\ 36.288.13.n.1,\ 36.288.13.p.1, 36.288.13.q.1,\ 36.288.13.r.1,\ 36.288.13.s.1,\ 36.288.13.t.1,\ 36.288.13.u.1, 36.288.13.v.1 \\ 
13 & 360 & 10.360.13.a.1,\ 10.360.13.b.1,\ 10.360.13.c.1,\ 10.360.13.d.1,\ 20.360.13.a.1, 20.360.13.b.1,\ 20.360.13.ba.1,\ 20.360.13.bb.1,\ 20.360.13.c.1,\ 20.360.13.d.1, 20.360.13.e.1,\ 20.360.13.f.1,\ 20.360.13.g.1,\ 20.360.13.h.1,\ 20.360.13.i.1, 20.360.13.j.1,\ 20.360.13.k.1,\ 20.360.13.l.1,\ 20.360.13.m.1,\ 20.360.13.n.1, 20.360.13.o.1,\ 20.360.13.p.1,\ 20.360.13.q.1,\ 20.360.13.r.1,\ 20.360.13.s.1, 20.360.13.t.1,\ 20.360.13.u.1,\ 20.360.13.v.1,\ 20.360.13.w.1,\ 20.360.13.x.1, 20.360.13.y.1,\ 20.360.13.z.1,\ 30.360.13.a.1,\ 30.360.13.b.1,\ 30.360.13.c.1, 30.360.13.d.1,\ 30.360.13.e.1,\ 30.360.13.f.1,\ 30.360.13.g.1,\ 30.360.13.h.1, 30.360.13.i.1,\ 30.360.13.j.1,\ 30.360.13.k.1,\ 30.360.13.l.1 \\ 
21 & 384 &  16.384.21.a.1,\ 16.384.21.a.2,\ 16.384.21.bf.1,\ 16.384.21.bg.1,\ 16.384.21.bk.1, 16.384.21.bl.1,\ 16.384.21.bn.1,\ 16.384.21.bo.1,\ 16.384.21.bq.1,\ 16.384.21.br.1, 16.384.21.k.1,\ 16.384.21.k.2,\ 16.384.21.k.3,\ 16.384.21.m.1,\ 16.384.21.p.1, 16.384.21.t.1,\ 16.384.21.t.2,\ 32.384.21.if.1,\ 32.384.21.iq.1,\ 32.384.21.iq.2,\ {\bf 32.384.21.eu.1, 32.384.21.ev.1, 32.384.21.js.1, 32.384.21.js.2.}\\ 
25 & 576 & 12.576.25.a.1,\ 12.576.25.b.1,\ 12.576.25.c.1,\ 12.576.25.d.1,\ 12.576.25.e.1, 12.576.25.e.2,\ 12.576.25.f.1,\ 12.576.25.g.1,\ 24.576.25.a.1,\ 24.576.25.b.1, 24.576.25.bg.1,\ 24.576.25.bh.1,\ 24.576.25.bi.1,\ 24.576.25.bj.1,\ 24.576.25.bj.2, 24.576.25.bk.1,\ 24.576.25.bk.2,\ 24.576.25.bl.1,\ 24.576.25.bm.1,\ 24.576.25.bn.1, 24.576.25.bo.1,\ 24.576.25.bo.2,\ 24.576.25.bp.1,\ 24.576.25.bp.2,\ 24.576.25.bq.1, 24.576.25.bq.2,\ 24.576.25.br.1,\ 24.576.25.br.2,\ 24.576.25.bs.1,\ 24.576.25.bs.2, 24.576.25.bt.1,\ 24.576.25.bt.2,\ 24.576.25.bu.1,\ 24.576.25.bu.2,\ 24.576.25.bv.1, 24.576.25.bv.2,\ 24.576.25.bw.1,\ 24.576.25.bw.2,\ 24.576.25.c.1,\ 24.576.25.cz.1, 24.576.25.d.1,\ 24.576.25.da.1,\ 24.576.25.db.1,\ 24.576.25.dc.1,\ 24.576.25.dd.1, 24.576.25.de.1,\ 24.576.25.e.1,\ 24.576.25.f.1,\ 24.576.25.f.2,\ 24.576.25.g.1, 24.576.25.h.1,\ 24.576.25.h.2,\ 36.576.25.a.1,\ 36.576.25.a.2,\ 36.576.25.b.1, 36.576.25.b.2,\ 36.576.25.c.1,\ 36.576.25.c.2,\ 36.576.25.d.1,\ 36.576.25.d.2, 36.576.25.e.1,\ 36.576.25.e.2,\ 36.576.25.e.3,\ 36.576.25.e.4,\ 36.576.25.f.1, 36.576.25.f.2,\ 36.576.25.g.1,\ 36.576.25.g.2,\ 36.576.25.h.1,\ 36.576.25.i.1, 36.576.25.j.1,\ 36.576.25.j.2,\ 36.576.25.k.1,\ 36.576.25.k.2;\\
26 & 660 & {\bf 11.660.26.a.1, 11.660.26.b.1, 11.660.26.c.1, 33.660.26.a.1 } \\
31 & 504 & 21.504.31.g.1,\ 21.504.31.h.1,\ 21.504.31.k.1;\\
41 & 768 & 16.768.41.g.1,\ 16.768.41.g.2,\ 16.768.41.g.3,\ 16.768.41.g.4,\ 16.768.41.h.1, 16.768.41.h.2,\ 16.768.41.h.3,\ 16.768.41.h.4,\ 16.768.41.p.1,\ 16.768.41.q.1, 16.768.41.r.1,\ 16.768.41.s.1,\ 16.768.41.t.1,\ 16.768.41.u.1,\ 32.768.41.fb.1, 32.768.41.fb.2,\ 32.768.41.fd.1,\ 32.768.41.fd.2;\\
45 & 672 & {\bf 28.672.45.ba.1, 28.672.45.bb.1, 28.672.45.bc.1, 28.672.45.x.1, 28.672.45.y.1, 28.672.45.z.1 } \\
49 & 1008 & {\bf 14.1008.49.a.1, 14.1008.49.b.1, 14.1008.49.c.1, 14.1008.49.d.1, 14.1008.49.e.1,\ 14.1008.49.f.1, 28.1008.49.a.1, 28.1008.49.b.1, 28.1008.49.c.1, 28.1008.49.d.1,\ 28.1008.49.e.1, 28.1008.49.f.1, 28.1008.49.g.1, 28.1008.49.h.1, 28.1008.49.i.1,\ 28.1008.49.j.1, 28.1008.49.k.1, 28.1008.49.l.1, 28.1008.49.m.1, 28.1008.49.n.1,\ 28.1008.49.o.1, 28.1008.49.p.1, 28.1008.49.q.1, 28.1008.49.r.1, 28.1008.49.s.1,\ 28.1008.49.t.1, 28.1008.49.u.1} \\
50 & 1092 & {\bf 13.1092.50.a.1, 13.1092.50.b.1, 13.1092.50.c.1, 13.1092.50.d.1, 39.1092.50.a.1,\   39.1092.50.b.1, 39.1092.50.c.1, 39.1092.50.d.1 } \\
73 & 1152 & 24.1152.73.cb.1,\ 24.1152.73.cb.2,\ 24.1152.73.ce.1,\ 24.1152.73.ce.2,\ 24.1152.73.db.1, 24.1152.73.db.2,\ 24.1152.73.de.1,\ 24.1152.73.de.2,\ {\bf 24.1152.73.a.1, 24.1152.73.b.1, 24.1152.73.baa.1, 24.1152.73.bab.1, 24.1152.73.bac.1,\ 24.1152.73.bad.1, 24.1152.73.bae.1, 24.1152.73.baf.1, 24.1152.73.bag.1, 24.1152.73.bah.1,\ 24.1152.73.bai.1, 24.1152.73.baj.1, 24.1152.73.bao.1, 24.1152.73.bap.1, 24.1152.73.baq.1,\ 24.1152.73.bar.1, 24.1152.73.bas.1, 24.1152.73.bat.1, 24.1152.73.bau.1, 24.1152.73.bav.1,\ 24.1152.73.baw.1, 24.1152.73.bax.1, 24.1152.73.bay.1, 24.1152.73.baz.1, 24.1152.73.bba.1,\ 24.1152.73.bbb.1, 24.1152.73.bbc.1, 24.1152.73.bbd.1, 24.1152.73.bbe.1, 24.1152.73.bbf.1,\ 24.1152.73.bbg.1, 24.1152.73.bbh.1, 24.1152.73.bbi.1, 24.1152.73.bbj.1, 24.1152.73.bbk.1,\ 24.1152.73.bbl.1, 24.1152.73.bbm.1, 24.1152.73.bbn.1, 24.1152.73.bbo.1, 24.1152.73.bbp.1,\ 24.1152.73.bbq.1, 24.1152.73.bbr.1, 24.1152.73.bbs.1, 24.1152.73.bbt.1, 24.1152.73.bbu.1,\ 24.1152.73.bbv.1, 24.1152.73.bbw.1, 24.1152.73.bbx.1, 24.1152.73.bby.1, 24.1152.73.bbz.1,\ 24.1152.73.bca.1, 24.1152.73.bcb.1, 24.1152.73.bcc.1, 24.1152.73.bcd.1, 24.1152.73.bce.1,\ 24.1152.73.bcf.1, 24.1152.73.bcg.1, 24.1152.73.bch.1, 24.1152.73.bci.1, 24.1152.73.bcj.1,\ 24.1152.73.by.1, 24.1152.73.cg.1, 24.1152.73.cg.2, 24.1152.73.cp.1, 24.1152.73.cp.2,\ 24.1152.73.cr.1, 24.1152.73.cr.2, 24.1152.73.ct.1, 24.1152.73.ct.2, 24.1152.73.ct.3,\ 24.1152.73.dx.1, 24.1152.73.eb.1, 24.1152.73.eb.2, 24.1152.73.ed.1, 24.1152.73.ed.2,\ 24.1152.73.ef.1, 24.1152.73.ef.2, 24.1152.73.eh.1, 24.1152.73.eh.2, 24.1152.73.eh.3,\ 24.1152.73.f.1, 24.1152.73.g.1, 24.1152.73.h.1, 24.1152.73.l.1, 24.1152.73.m.1,\ 24.1152.73.n.1, 24.1152.73.o.1, 24.1152.73.p.1, 24.1152.73.q.1, 24.1152.73.r.1,\ 24.1152.73.rr.1, 24.1152.73.rv.1, 24.1152.73.rv.2, 24.1152.73.rx.1, 24.1152.73.rx.2,\ 24.1152.73.rz.1, 24.1152.73.rz.2, 24.1152.73.s.1, 24.1152.73.sb.1, 24.1152.73.sb.2,\ 24.1152.73.sb.3, 24.1152.73.t.1, 24.1152.73.tp.1, 24.1152.73.tt.1, 24.1152.73.tt.2,\ 24.1152.73.tv.1, 24.1152.73.tv.2, 24.1152.73.tx.1, 24.1152.73.tx.2, 24.1152.73.tz.1,\ 24.1152.73.tz.2, 24.1152.73.tz.3, 24.1152.73.u.1, 24.1152.73.v.1, 24.1152.73.w.1,\ 24.1152.73.wm.1, 24.1152.73.wn.1, 24.1152.73.wp.1, 24.1152.73.wq.1, 24.1152.73.ws.1,\ 24.1152.73.wt.1, 24.1152.73.wu.1, 24.1152.73.wv.1, 24.1152.73.wy.1, 24.1152.73.wz.1,\ 24.1152.73.x.1, 24.1152.73.xa.1, 24.1152.73.xb.1, 24.1152.73.xc.1, 24.1152.73.xd.1,\ 24.1152.73.xe.1, 24.1152.73.xf.1, 24.1152.73.xk.1, 24.1152.73.xl.1, 24.1152.73.xo.1,\ 24.1152.73.xp.1, 24.1152.73.xq.1, 24.1152.73.xr.1, 24.1152.73.xs.1, 24.1152.73.xt.1,\ 24.1152.73.xw.1, 24.1152.73.xx.1, 24.1152.73.xy.1, 24.1152.73.xz.1, 24.1152.73.y.1,\ 24.1152.73.ya.1, 24.1152.73.yb.1, 24.1152.73.yc.1, 24.1152.73.yd.1, 24.1152.73.ye.1,\ 24.1152.73.yf.1, 24.1152.73.yg.1, 24.1152.73.yh.1, 24.1152.73.yi.1, 24.1152.73.yj.1,\ 24.1152.73.yk.1, 24.1152.73.yl.1, 24.1152.73.ym.1, 24.1152.73.yn.1, 24.1152.73.yo.1,\ 24.1152.73.yp.1, 24.1152.73.yq.1, 24.1152.73.yr.1, 24.1152.73.ys.1, 24.1152.73.yt.1,\ 24.1152.73.z.1, 24.1152.73.zu.1, 24.1152.73.zv.1, 24.1152.73.zw.1, 24.1152.73.zx.1,\ 24.1152.73.zy.1, 24.1152.73.zz.1.}\\ 
73 & 1440 & {\bf 15.1440.73.a.1, 15.1440.73.b.1. }\\
81 & 1320 & {\bf 22.1320.81.b.1, 22.1320.81.c.1, 22.1320.81.d.1, 22.1320.81.e.1, 22.1320.81.f.1, 22.1320.81.h.1 } \\
81 & 1536 & {\bf 16.1536.81.a.1, 16.1536.81.a.2, 16.1536.81.a.3, 16.1536.81.a.4, 16.1536.81.b.1,\ 16.1536.81.b.2, 16.1536.81.b.3, 16.1536.81.b.4, 16.1536.81.c.1, 16.1536.81.c.2,\ 16.1536.81.c.3, 16.1536.81.c.4, 16.1536.81.d.1, 16.1536.81.d.2, 16.1536.81.d.3,\ 16.1536.81.d.4, 32.1536.81.de.1, 32.1536.81.df.1, 32.1536.81.dq.1, 32.1536.81.dr.1,\ 32.1536.81.ds.1, 32.1536.81.ds.2, 32.1536.81.ds.3, 32.1536.81.ds.4, 32.1536.81.eb.1,\ 32.1536.81.ec.1}\\
106 & 1320 & {\bf  22.1980.106.a.1, 22.1980.106.b.1, 22.1980.106.c.1. }\\
121 & 2016 & {\bf  21.2016.121.a.1, 21.2016.121.b.1, 21.2016.121.c.1, 21.2016.121.d.1, 21.2016.121.e.1,\  21.2016.121.f.1 }\\
133 & 2016 & {\bf  28.2016.133.bc.1, 28.2016.133.bd.1, 28.2016.133.be.1, 28.2016.133.bf.1, 28.2016.133.bg.1,\    28.2016.133.bh.1, 28.2016.133.e.1, 28.2016.133.l.1, 28.2016.133.m.1, 28.2016.133.t.1,\  28.2016.133.v.1, 28.2016.133.w.1 }\\
133 & 2448 & {\bf  17.2448.133.a.1, 17.2448.133.b.1, 17.2448.133.c.1, 17.2448.133.d.1, 17.2448.133.e.1 }\\
141 & 2184 & {\bf  26.2184.141.c.1, 26.2184.141.d.1, 26.2184.141.f.1, 26.2184.141.g.1, 26.2184.141.i.1,\   26.2184.141.j.1, 26.2184.141.l.1, 26.2184.141.n.1}\\
196 & 3420 & {\bf  19.3420.196.a.1, 19.3420.196.b.1, 19.3420.196.c.1, 19.3420.196.d.1, 19.3420.196.e.1. }\\
209 & 3072 & {\bf  32.3072.209.br.1, 32.3072.209.br.2, 32.3072.209.br.3, 32.3072.209.br.4, 32.3072.209.bv.1,\ 32.3072.209.bv.2, 32.3072.209.bv.3, 32.3072.209.bv.4, 32.3072.209.cl.1, 32.3072.209.cn.1,\ 32.3072.209.co.1, 32.3072.209.cr.1, 32.3072.209.u.1, 32.3072.209.u.2, 32.3072.209.u.3,\ 32.3072.209.u.4.}\\
\hline
\end{longtable}

In Table \ref{tab: non-modular aut} we list all modular curves of level $n\le 39$ that have a hyperelliptic model in the LMFDB and admit a non-modular automorphism.

\begin{longtable}{c c p{0.8\textwidth}} 
\caption{Hyperelliptic curves with non-modular automorphisms}
\label{tab: non-modular aut}
\\
\hline
$\frac{|\mathrm{Aut}(X_H)|}{|\mathrm{MAut}(X_H)|}$ & Level & Curves (LMFDB label) \\
\hline
\\
\endhead                                                        
\hline
\endfoot            
	\\
	\hline
	\\
	2 & 10 & 10.30.2.a.1,\ 10.30.2.b.1,\ 10.30.2.c.1,\ 10.30.2.d.1,\ 10.60.2.d.1,\ 10.60.2.e.1\\
	& 12 & 12.24.2.a.1,\ 12.24.2.b.1,\ 12.24.2.c.1,\ 12.24.2.d.1,\ 12.48.3.c.1,\ 12.48.3.f.1,\ 12.48.3.h.1,\ 12.48.3.i.1,\ 12.48.3.l.1,\ 12.48.3.m.1,\ 12.48.3.r.1,\ 12.48.3.t.1\\
	& 15 & 15.30.2.a.1,\ 15.30.2.b.1,\ 15.30.2.c.1,\ 15.30.2.d.1,\ 15.60.2.a.1,\ 15.60.2.b.1,\ 15.60.2.c.1,\ 15.60.2.d.1\\
	& 16 & 16.48.3.bm.1,\ 16.48.3.bn.1,\ 16.48.3.bo.1,\ 16.48.3.bq.1,\ 16.48.3.br.1,\ 16.48.3.bs.1,\ 16.48.3.bx.1,\ 16.48.3.bz.1,\ 16.48.3.ca.1,\ 16.48.3.cb.1,\ 16.48.3.cd.1,\ 16.48.3.ce.1,\ 16.96.3.ey.1,\ 16.96.3.ey.2,\ 16.96.3.fa.1,\ 16.96.3.fa.2\\
	& 18 & 18.24.2.a.1,\ 18.24.2.b.1,\ 18.24.2.c.1,\ 18.72.2.a.1,\ 18.72.2.c.1,\ 18.72.2.c.2\\
	& 20 & 20.30.2.a.1,\ 20.30.2.b.1,\ 20.30.2.c.1,\ 20.30.2.d.1,\ 20.30.2.e.1,\ 20.30.2.f.1,\ 20.30.2.g.1,\ 20.30.2.h.1,\ 20.40.2.a.1,\ 20.40.2.b.1,\ 20.40.2.c.1,\ 20.40.2.d.1,\ 20.40.2.e.1,\ 20.60.2.d.1,\ 20.60.2.e.1,\ 20.60.2.f.1,\ 20.60.2.g.1,\ 20.60.2.h.1,\ 20.60.2.i.1,\ 20.80.3.b.1,\ 20.80.3.d.1\\
	& 21 & 21.24.2.a.1,\ 21.48.3.a.1,\ 21.48.3.b.1,\ 21.48.3.c.1,\ 21.48.3.d.1\\
	& 24 & 24.24.2.a.1,\ 24.24.2.b.1,\ 24.24.2.c.1,\ 24.24.2.d.1,\ 24.72.2.hk.1,\ 24.72.2.hk.2,\ 24.72.2.hl.1,\ 24.72.2.hl.2,\ 24.72.2.hm.1,\ 24.72.2.hm.2,\ 24.72.2.hn.1,\ 24.72.2.hn.2,\ 24.48.3.ba.1,\ 24.48.3.bb.1,\ 24.48.3.bc.1,\ 24.48.3.be.1,\ 24.48.3.bf.1,\ 24.48.3.bj.1,\ 24.48.3.bm.1,\ 24.48.3.bn.1,\ 24.48.3.bo.1,\ 24.48.3.bw.1,\ 24.48.3.bx.1,\ 24.48.3.cb.1,\ 24.48.3.cc.1,\ 24.48.3.cd.1,\ 24.48.3.ch.1,\ 24.48.3.i.1,\ 24.48.3.j.1,\ 24.48.3.k.1,\ 24.48.3.m.1,\ 24.48.3.n.1,\ 24.48.3.r.1,\ 24.48.3.s.1,\ 24.48.3.t.1,\ 24.48.3.x.1\\
	& 25 & 25.50.2.a.1,\ 25.75.2.a.1\\
	& 27 & 27.36.2.b.1\\
	& 28 & 28.32.2.a.1,\ 28.64.3.b.1\\
	& 30 & 30.30.2.a.1,\ 30.30.2.b.1,\ 30.30.2.c.1,\ 30.30.2.d.1,\ 30.30.2.e.1,\ 30.30.2.f.1,\ 30.30.2.g.1,\ 30.30.2.h.1,\ 30.60.2.d.1,\ 30.60.2.f.1\\
	& 36 & 36.24.2.a.1,\ 36.24.2.b.1\\
	& 37 & 37.38.2.a.1\\
	& 39 & 39.42.2.a.1\\
	3 &16 & 16.24.2.a.1,\ 16.24.2.a.2,\ 16.24.2.b.1,\ 16.24.2.b.2\\
	4 &11& 11.66.2.a.1\\
	&16 & 16.64.2.a.1\\
	& 30 & 30.60.2.g.1,\ 30.60.2.h.1,\ 30.60.2.i.1,\ 30.60.2.j.1\\
	\hline
\end{longtable}

\end{document}